\documentclass[12pt]{amsart}

\usepackage{amsmath,amsfonts,amsthm,amsopn,cite,mathrsfs}
\usepackage{epsfig,verbatim}
\usepackage{subfigure}
\include{macros}

\usepackage{verbatim}
\usepackage{mathrsfs}
\usepackage[pdftex,
colorlinks, 
bookmarksnumbered, 
bookmarksopen, 
linkcolor=blue, 
citecolor=blue, 
]{hyperref}

\usepackage[latin1]{inputenc}

\usepackage{amssymb}
\usepackage{amsthm}
\usepackage{amsfonts}
\usepackage{latexsym}
\usepackage{amsmath}
\usepackage{mathrsfs}
\usepackage{xfrac}
\usepackage{mathabx}
\usepackage{mathdots}
\usepackage{graphicx}

\usepackage{url}

\newtheorem{thm}{Theorem}\numberwithin{thm}{section}
\newtheorem{dfn}[thm]{Definition}
\newtheorem{prop}[thm]{Proposition}
\newtheorem{lem}[thm]{Lemma}
\newtheorem{cor}[thm]{Corollary}

\newtheoremstyle{plainNoItalics}{}{}{\normalfont}{}{\bfseries}{.}{ }{}
\theoremstyle{plainNoItalics}
\newtheorem{ex}[thm]{Example}

\newtheorem{rem}[thm]{Remark}

\newcommand{\comms}[1]{{\color{red} #1}}

\newcommand{\R}{\mathbb{R}}
\newcommand{\N}{\mathbb{N}}
\newcommand{\CF}{\mathbb{C}}
\newcommand{\Q}{\mathbb{Q}}

\renewcommand{\l}{\lambda}
\newcommand{\Norm}[2]{\left\| #2\right\|_{#1}}
\renewcommand{\L}[2]{L^{#1}(#2)}

\def\rd{\R^d}
\def\lrd{L^2(\rd)}
\newcommand{\fd}{d_\pi}
\newcommand{\Pf}{\mathrm{Pf}}

\newcommand{\onb}{orthonormal basis}
\newcommand{\onbs}{orthonormal bases}
\newcommand{\rep}{representation}
\renewcommand{\H}{\mathbf{H}_d}
\newcommand{\Z}{\mathbb{Z}}

\newcommand{\zd}{\mathbb{Z}^d}

\newcommand{\bracket}[2]{\left\langle #1 , #2\right\rangle}

\newcommand{\nn}{\nonumber}
\newcommand{\h}{\mathfrak{h}_d}

\newcommand{\Abs}[1]{\left| #1\right|}
\newcommand{\subbracket}[3]{{\left\langle #2, #3\right\rangle}_{ #1}}

\newcommand{\id}{\mathop{\mathrm{id}}}

\newcommand{\HG}[2]{\mathbf{H}_{#2, #1}}
\newcommand{\HA}[2]{\mathfrak{h}_{#1, #2}}
\newcommand{\inv}{^{-1}}

\newcommand{\Lie}[1]{\frak{#1}}
\newcommand{\ggl}{\Lie{g}}

\newcommand{\pid}{\mathfrak{m}}
\newcommand{\qa}{\mathfrak{h}}
\newcommand{\Liez}[1]{\mathfrak{z}(\mathfrak{#1})}
\newcommand{\PID}{M}
\newcommand{\QA}{H}
\newcommand{\subgr}{\leq}
\newcommand{\nsubgr}{\vartriangleleft}
\newcommand{\dimG}{n}
\newcommand{\rquo}{\setminus}

\newcommand{\Orbit}{\mathcal{O}}
\newcommand{\cO}{\mathcal{O}}

\newcommand{\cH}{\mathcal{H}}
\newcommand{\HS}{\mathcal{H}}
\newcommand{\RS}{\mathcal{H}_{\pi}}

\newcommand{\Rspan}[1]{\mathbb{R}\text{-}\mathrm{span}\{ #1\}}
\newcommand{\Qspan}[1]{\mathbb{Q}\text{-}\mathrm{span}\{ #1\}}

\newcommand{\Gmeas}[2]{\mu_{#1}(#2)}

\newcommand{\Ad}{\mathop{\mathrm{Ad}}}
\newcommand{\coAd}{\mathop{\mathrm{Ad}^*}}

\newcommand{\coad}{\mathop{\mathrm{ad}^*}}

\newcommand{\indR}[3]{\mathop{\mathrm{ind}^{#3}_{#2}(#1)}}

\newcommand{\PFD}{F} 
\newcommand{\pr}{\mathop{\mathrm{pr}}}
\newcommand{\Aut}{\mathop{\mathrm{Aut}}}
\newcommand{\dil}{\delta}

\usepackage{graphicx}

\makeatletter
\newcommand*\bigcdot{\mathpalette\bigcdot@{.5}}
\newcommand*\bigcdot@[2]{\mathbin{\vcenter{\hbox{\scalebox{#2}{$\m@th#1\bullet$}}}}}
\makeatother

\begin{document}
\begin{abstract}

Let $G$ be a connected, simply connected nilpotent group and $\pi$
be a square-integrable
 irreducible unitary representation modulo its center $Z(G)$ on $\L{2}{\R^d}$. We prove
that under reasonably weak conditions on $G$ and $\pi$ there exist a
discrete subset $\Gamma$ of $G/Z(G)$  and some (relatively) compact set $\PFD \subseteq
\R^d$
such that
	\begin{align*} 
		\bigl\{ \Abs{F}^{-1/2} \hspace{2pt} \pi(\gamma) 1_\PFD \mid \gamma \in \Gamma \bigr\}
	\end{align*}
forms an orthonormal basis of $\L{2}{\R^d}$. This construction
generalizes the well-known example of Gabor orthonormal bases in
time-frequency analysis.

The main theorem covers graded Lie groups with one-dimensional
center. In the presence of a rational structure, the set $\Gamma $ can
be chosen to be a uniform subgroup of $G/Z$. 
\end{abstract}

\title[Orthonormal Bases and Nilpotent Lie Groups]{Orthonormal Bases
in the Orbit of Square-Integrable Representations of Nilpotent Lie
  Groups} 

\author{Karlheinz Gr\"ochenig}
\address{Faculty of Mathematics \\
University of Vienna \\
Oskar-Morgenstern-Platz 1 \\
A-1090 Vienna, Austria}
\email{karlheinz.groechenig@univie.ac.at}
\subjclass[2000]{}
\date{}

\author{David Rottensteiner}
\address{Faculty of Mathematics \\
University of Vienna \\
Oskar-Morgenstern-Platz 1 \\
A-1090 Vienna, Austria}
\email{david.rottensteiner@univie.ac.at}

\thanks{K.\ G.\ and D.\ R. were
  supported in part by the  project P26273 - N25  of the
Austrian Science Fund (FWF)}

\maketitle

\section{Introduction}

Haar, Gabor, wavelet --- these names all refer to structured
orthonormal bases for $L^2(\R ^d)$. So far these have been objects of
applied harmonic analysis, because they are useful to distinguish and
classify certain features of functions, such as smoothness~\cite{meyer-1} or
phase space localization~\cite{fg97jfa}. From an abstract point of view their common
origin lies in the representation theory of certain groups. Each of
these orthonormal bases is in the orbit of a unitary representation
of a locally compact group. Specifically, the Gabor basis arises by
applying phase space shifts to the characteristic function $1_{[0,1]}$
to yield the \onb\ $\{ e^{2\pi i l x} 1_{[0,1]}(x-k) \mid  k,l \in \Z \}$
of $L^2(\R)$; the Haar basis arises by applying shifts and dilations
to a  ``wavelet'', e.g., the function $ \psi (x)= 1_{[0,1/2]}(x) -
1_{[1/2,1]}(x)$  yields the \onb $\{ 2^{j/2} \psi (2^jx-k) \mid  j,k \in \Z \}$ 
 of $L^2(\R)$. 

The underlying question in representation theory is the following.
 Let $
G$ be a locally compact group and $\pi $ be a unitary
representation of $G$ on a Hilbert space $\HS $. Under which
conditions do there exist a template $f \in \HS $ and a discrete subset
$\Gamma \subseteq G$, such that $\{ \pi (\l )f \mid \gamma \in
\Gamma \}$ is an \onb\ for $\HS $? In other words, do there exist
orbits of the representation that contain an
\onb ?

To our knowledge the existence problem seems to be a new question in
abstract harmonic analysis. Beyond the special cases that are studied
for applications (wavelet bases, Gabor bases) no results seem to be
known, let alone general results. 

Wavelet theory deals with the construction of \onbs\  in the orbit of
the quasi-regular representation of certain solvable  products,
usually  a semidirect product of $\rd  $ with some matrix action~\cite{GM92, GLWW04}.
A special aspect of Gabor analysis is the existence and construction of
\onbs\ in the orbit of the Schr\"odinger representation of the
Heisenberg group~\cite{FollPhSp}. To our knowledge, certain decomposition results of 
representations of a semisimple Lie group restricted to  a
lattice subgroup yield  the existence of an \onb\ in the orbit of
discrete series representations. This was  pointed out to us by Yurii
Neretin~\cite{neretin} and is also implicit in Bekka's work~\cite{bekka04}.

In this paper we study the existence of \onbs\ in the orbit of
irreducible unitary representations of a nilpotent Lie group. As it
is not difficult to see that a necessary condition for the existence
of an \onb\ (or a frame) in the orbit of a representation is the
square-integrability modulo a suitable subgroup~\cite{ancalevi01}, we
will restrict our attention to irreducible unitary representations of
a connected, simply connected nilpotent Lie group that is square-integrable
modulo its center. So our precise question is as follows: 

\emph{Let $G$ be a connected,  simply connected nilpotent Lie group and $(\pi , \HS )$
be an irreducible unitary \rep\ of $G$ that is square-integrable modulo
the center of $G$ (we write $\pi \in SI/Z(G)$). Does there
exist a set  $\Gamma \subseteq G/Z(G) $ and a vector $f \in \HS$,
such that $\{ \pi(\gamma )f \mid \gamma \in \Gamma \}$ is an \onb\ of
$\HS$?}  

The answer is obviously yes for the Heisenberg group $\H$ and the
(Schr\"odinger) representations $\pi _\l (x,y,z) f(t) = e^{2\pi i \l z}
  e^{2\pi \l y t} f(t + x)$ on $\L{2}{\R ^d}$,  because $\{\pi
  _\l (k,\l \inv l, 0) 1_{[0,1]}(t)\mid k,l\in \Z^d \}$  is an \onb\ for
  $\L{2}{\R}$. For low-dimensional nilpotent Lie groups the question
  has been studied in A.~H\"ofler's thesis~\cite{Hoe14}. She found that
  the answer is yes for all nilpotent Lie  groups $G$ of dimension
  $\dim(G) \leq 6$ and all irreducible \rep s in $SI/Z(G)$. The
  analysis of 
  explicit examples in low dimension justifies the conjecture that the
  answer to our question should be  affirmative for all nilpotent
  Lie groups  and
  \rep s in $SI/Z $, or at least for a significant, large class of
  nilpotent groups. However, as H\"ofler's work is based on  the
  case-by-case inspection of the 
  explicit formulas for the irreducible representations
  in~\cite{Nielsen}, her work gives little indication on what might
  happen for nilpotent groups in higher dimensions. 

In this paper we 
prove a general result about \onbs\ in the orbit of a   square-integrable \rep\
of a nilpotent group. Its spirit is captured in the following special
case. 

\vspace{3mm}

\noindent \textbf{Theorem $0$.} \label{psth}
\emph{  Let $G$ be a graded Lie group with one-dimensional center and $(\pi , \HS ) $ be a
 square-integrable irreducible unitary representation modulo center, $\pi \in SI/Z(G)$.
 Then there exist a discrete
 subset $\Gamma \subseteq G/Z(G)$ and an $f \in \HS $ such that 
the set $\{ \pi (\gamma ) f \mid \gamma \in \Gamma \}$ is an \onb\ of $\HS
$. }

\vspace{3mm}

A more detailed formulation and the proof require the full details  of Kirillov's
construction of irreducible representations via the coadjoint orbits
and thus a non-negligeable amount of preparation.  Other and more
general formulations will be given in Theorems~\ref{MainThmQL},
\ref{MainThmUnif}, and \ref{TVMainThm}. 

Roughly, an irreducible unitary representation of $G$ is parametrized
by a linear functional $l$ acting on the Lie algebra $\Lie{g}$ of $G$
and by a certain subgroup $M$ of $G$, a so-called polarization. The
imposed conditions depend (a) on the structure of $G$, for instance,
we assume that some polarization $M$ is a normal subgroup of $G$ and
that 
we know some crucial facts  about the coadjoint action, and (b) some 
rationality assumptions, for instance 
$\Lie{g}$ has rational   structure coefficients  and  the functional $l$ also has rational
coefficients with respect to the  chosen basis.

(i) The set $\Gamma $ in  Theorem $0$  is not arbitrary. Under
some  rationality assumptions the construction yields a uniform subgroup $\Gamma
$ in $G/Z(G)$ (Theorem~\ref{MainThmUnif} and~\ref{TVMainThm}).
However, we will show that the existence of orthonormal bases is not
tied to the rational structure of the group. Indeed, we will prove  a version of
Theorem $0$ by using a quasi-lattice in $G$ in place of a lattice. Such 
structures exist in arbitrary nilpotent Lie groups~\cite{FGr}.


(ii) In Theorem $0$   we  make heavy use of  the description and
construction of representations with 
 Kirillov's  method of coadjoint orbits.
The \rep\ space $\RS$ is then realized  by $L^2(\rd
)$;  the representation is essentially of the form $\bigl( \pi(x)f \bigr) (\zeta ) =
e^{2\pi i P( x,\zeta )} f(\zeta \cdot x)$, where $P$ is a polynomial
of $x$ and $\zeta$,  and $x$ acts on $\zeta \in \R^d $.
Thus the main theorems prove the existence of a new class of \onbs\
for $L^2(\rd )$. We believe that these bases could be useful in the
non-Euclidean analysis of $\rd $ (viewed as a homogeneous space).  

(iii) Whereas in the low-dimensional examples this action $\zeta \to \zeta \cdot x$ is
always Abelian and roughly a translation $\zeta \cdot x \simeq \zeta +
x$ on $\rd $, this action is generally non-Abelian in higher
dimensions. We will demonstrate  this
new aspect with the example of the meta-Heisenberg group. 

(iv) The construction of the orthonormal basis actually happens in
$G/Z(G)$. Clearly, if $z_\gamma \in Z(G)$ is arbitrary and $\{ \pi
(\gamma ) g \mid  \gamma \in G \}$ is an orthonormal basis for $\cH$, then
so is $\{ \pi (z_\gamma \gamma ) g \mid \gamma \in G \}$. We could
therefore investigate the problem for the projective representations
of the nilpotent group $G/Z(G)$. However, since we use  the Kirillov
construction of irreducible representations, we prefer to work with
the representations of $G$ and accept the ambiguity in the phase
rather than passing to $G/Z(G)$.

(v) In $\HS = L^2(\rd )$ the template  $f$ generating the  orthonormal
orbit is a characteristic function
$f = 1_F $ of a suitable fundamental domain in $\PID \rquo G$. This raises the natural
question whether there are smooth $f \in L^2(\rd )$ that generate \onbs.
This problem has been studied extensively for the Heisenberg group
where the answer is a decisive ``no''. Precisely, if $\{ e^{2\pi i l
  t} g(t + k) \mid k,l \in \Z \}$ is an \onb\ for $L^2(\R )$, then either
$x \mapsto xf(x) \not \in L^2(\R )$ or $f' \not \in L^2(\R )$. This is the
famous Balian-Low theorem, which has inspired a whole direction of
harmonic analysis~\cite{behewa95,CzPow}. Since Kirillov's lemma asserts that every
nilpotent Lie group contains a subgroup isomorphic to the Heisenberg
group, we expect that a Balian-Low type theorem also holds in
the context of  Theorem $0$. Indeed, it is possible to
prove a version of the Balian-Low theorem for arbitrary homogeneous
nilpotent groups. The proof is almost identical to the proof of the
Balian-Low theorem via the theory of deformation of frames
in~\cite{GOR15} (with only technical modifications). We will exploit
this direction in future work.

Our own motivation for the topic comes from studying the density of frames for
representations of nilpotent groups. It has been known for a
long time that frames always exist in the orbit of every square-integrable
irreducible unitary representation of a  locally compact group \cite{fegr89, gr91}.  In analogy to the sampling and
interpolation theory for bandlimited functions one would like to
understand how dense such a subset must necessarily be. Until
recently, all proofs of density theorems required the existence of an
\onb\ and the comparison of the frame to this \onb . See~\cite{bacahela06}
for a very general density theory of frames. For frames in the orbits
of nilpotent groups precise results have been derived by
H\"ofler~\cite{Hoe14}, but again under the hypothesis that there exists an
\onb\ in the orbit of the representation. With 
 Theorem $0$ and its variations,  H\"ofler's results about the
density of frames now apply to a general class of nilpotent groups. 
Let us mention that these density results have been generalized
recently to arbitrary unimodular groups with square-integrable
representations, but \emph{without} requiring the existence of an
orthonormal basis in an orbit~\cite{FGHKR}. 

Outlook: Our investigation raises several new questions in the theory
of nilpotent Lie groups. It is certainly tempting to believe that
every square-integrable irreducible unitary representation modulo the
center possesses an \onb\ in some orbit. However, at this time we are
far from  a complete answer.  We do not need a rational structure, as one might
expect from examples and Bekka's work~\cite{bekka04}, but we need that
the representation is induced from a normal polarizing subgroup and we
need special Malcev bases  that  are compatible with the
parametrization of the  coadjoint action (we call them
Chevalley-Rosenlicht admissible).  
The normality of a polarizing subgroup may not be necessary, as is
shown by an explicit construction of F\"uhr~\cite{fuehr16} of an \onb\ for an $11$-dimensional
nilpotent group with square-integrable representations that do not
possess a normal polarizing subgroup.  As for the existence of
Chevalley-Rosenlicht admissible
Malcev bases,  we do not  know for which groups they exist.

The paper is organized as follows: Section~2 is a warm-up and presents
the example of the Heisenberg group and the Gabor basis. 
 In Section~3 we give a quick review
of Malcev bases of nilpotent groups and of graded nilpotent Lie algebras; we set the notation for induced
representations, and review the standard
construction of irreducible representations of nilpotent Lie
groups. Section~4 contains  several  formulations of
Theorem $0$ and builds up the arguments  towards its proof. In
particular, we prove the existence of an orthonormal basis in the orbit
of a square-integrable representation for  all graded nilpotent Lie
groups with one-dimensional center. 
Section~5 offers concrete examples and treats  the 
meta-Heisenberg group in detail.

\textbf{Acknowledgements.} We would like to thank Hartmut F\"uhr and
Yuri Neretin for valuable discussions. K.\ G.\ would like to thank the 
Hausdorff Research Institute for Mathematics, Bonn, for the perfect
conditions during the work on this paper.

\section{Warm-Up}

In this section we will  review the example of the Gabor basis in $L^2(\rd
)$ and the Heisenberg group,   as the  proof of our main result is
modeled  on the arguments of this example.

\subsection{Gabor Bases and the Heisenberg Group}

For $\l \neq 0$, let $\pi_\l $  be the operator
	\begin{align} \label{eq:c11}
		(\pi_\l(z, y, x) f)(t) = e^{2 \pi i \l z} e^{2 \pi i \l yt} f(t+x),
	\end{align}
acting on  $f \in \lrd $.

The following statement is well-known in time-frequency analysis and
yields an orthonormal basis for $\lrd $ in the orbit of a single
function. 
(See \cite[\S~6.4]{gr01} and  \cite[\S~11.2]{he11}.)

	\begin{ex} \label{MainThmEuclVersion}
The family of time-frequency shifted functions
	\begin{align*}
		\Abs{\l}^{-d/2} \pi_\l(0,\vartheta, \gamma) 1_{[0,
                  \l^{-1})^d} = \Bigl( t \mapsto \Abs{\l}^{d/2} e^{2
                  \pi i \l \vartheta t} \hspace{2pt} 1_{[0,
                  \l^{-1})^d}(t +\eta) \Bigr), 
	\end{align*}
for $(\vartheta, \eta) \in  \Z^d \times \l^{-1} \Z^d$, forms an orthonormal basis for $\L{2}{\R^d}$.
	\end{ex}

	\begin{proof}
We partition $\R^d$ into the cubes $\Omega_\eta
:= [0, \l^{-1})^d + \eta$ with $\eta \in \l^{-1} \Z^d$. Then for arbitrary but
fixed $\eta \in \l^{-1} \Z^d$ the set of functions $ \{|\lambda
|^{d/2} e^{2 \pi i \l \vartheta t} \mid \vartheta \in \Z^d\}$ forms an
orthonormal basis for $\L{2}{[0, \l^{-1})^d + \eta}$ (the standard
Fourier basis). Including shifts, 
we find that
	\begin{align*}
		\Bigl \langle \pi_\l(0,\vartheta, \eta) 1_{[0, \l^{-1})^d}, \pi_\l(0,\vartheta', \eta') &1_{[0, \l^{-1})^d} \Bigr \rangle_{\L{2}{\R^d}} & \\
		&= \int_{\R^d} e^{2 \pi i \l \vartheta t} \hspace{2pt} 1_{[0, \l^{-1})^d}(t + \eta) \hspace{2pt} e^{-2 \pi i \l \vartheta' t} \hspace{2pt} 1_{[0, \l^{-1})^d}(t + \eta') \,dt \\
		&= \delta_{\eta, \eta'} \int_{\R^d} e^{2 \pi i \l (\vartheta - \vartheta') t} \hspace{2pt} 1_{[0, \l^{-1})^d}(t + \eta) \,dt \\
		&= \delta_{\eta, \eta'} \int_{[0, \l^{-1})^d} e^{2 \pi i \l (\vartheta - \vartheta') (t - \eta)} \,dt \\
		&= \delta_{\eta, \eta'} \hspace{2pt} e^{-2 \pi i \l (\vartheta - \vartheta') \eta} \int_{[0, \l^{-1})^d}
		e^{2 \pi i \l (\vartheta - \vartheta') t} \,dt \\
		&= \delta_{\vartheta, \vartheta'} \delta_{\eta, \eta'} \hspace{2pt} \hspace{2pt} \Abs{\l}^{-d}
	\end{align*}
holds true for all $(\vartheta, \eta), (\vartheta', \eta') \in \Z^d \times \l^{-1} \Z^d$. Since the completeness is
clear, the set  $\{\Abs{\l}^{-d/2} \pi_\l(0,\vartheta, \eta)
1_{[0,\l^{-1})^d} | (\vartheta,   \eta) \in \zd \times \lambda \inv \zd \}$ forms an orthonormal basis for $\lrd$. 
	\end{proof}

A similar construction yields \onbs\  for $L^2(A)$ of an arbitary
locally compact Abelian group $A$~\cite{grst07}.  

In the language of abstract harmonic analysis, $\pi _\l $ is a
unitary representation of 
the  $2d+1$-dimensional
Heisenberg group   $\H = \R \times \rd \times \rd $ with center $Z(\H ) = \R \times \{0\}
\times \{0\}$ and multiplication 
$$
	(z, y, x) (z', y', x') = (z+z' + x y', y+y',x+x) 
$$
for $z,z' \in \R, x,x', y,y'\in \rd $. 
The representation $\pi _\l $ for $\l \neq 0$  is 
the Schr\"{o}dinger representation and is square-integrable modulo
center, in short $\pi _\l \in SI/Z(\H)$.

We make the following observations:

(a) The variable $z$ in~\eqref{eq:c11} (the center of $\H $) does not play
any role in the definition of the \onb . In general, the construction
of an \onb\ in the orbit of a representation depends only on the
action of $G/Z(G)$ and not on the full representation.

(b) The set
$\Gamma := \l \inv \Z \times \l \inv  \zd  \times   \zd
\subset \H $ used in the example  
forms a uniform subgroup of $\H $, i.e., a discrete co-compact
subgroup of $\H$. The quotient subgroup $\Gamma/Z(\H) $ is isomorphic
to the lattice $ \Z^d \times \l^{-1}  \Z^d$. Furthermore, the set 
$[0,\l^{-1})^d$ is a fundamental domain for the lattice $\l^{-1} \zd $ in $\rd$. 

(c) The Heisenberg group is a semidirect product $\PID
\rtimes \QA$ of the normal subgroup $\PID = \{(z,y,0) \mid z \in \R, y
\in \rd\}$ and the subgroup $\QA = \{ (0,0,x) \mid x \in \rd \}$ (in the
jargon: of a polarizing normal subgroup $\PID$ and the
subgroup $\QA \cong \PID \rquo \H$). The variables in $\PID$ act by
modulation, whereas the variables in $\QA$ act by translation.

To generalize this elementary construction of \onbs\ on $\H $ to other nilpotent
groups, we need to find an appropriate coordinate system and
representations of $G$ in which
the variables split  into some form of modulation and translation
operators.

\section{Some Tools from the Theory of Nilpotent Lie Groups} \label{SectionTools}

Throughout the paper we will make heavy use of  the theory of
nilpotent Lie groups and their irreducible unitary representations. Our 
main reference is  the outstanding monograph of Corwin and Greenleaf
\cite{CorwinGreenleaf}.  We briefly recall some well-known
facts about strong Malcev bases, rational structures, induced representations, square-integrable representations and graded groups.
For more details confer Corwin and
Greenleaf~\cite[Ch.~1]{CorwinGreenleaf} as well as Fischer and Ruzhansky~\cite[\S~3.1.1]{FiRuz}.
To make the paper accessible also to applied harmonic analysts, we
present these notions in more detail than is perhaps necessary for
specialists. 

In the following $G$ is always a connected, simply connected nilpotent Lie group with Lie algebra
$\Lie{g}$. The vector space dual of $\Lie{g}$ is denoted by
$\Lie{g}^*$, the Haar measure of $G$ is $\mu _G$.

\subsection{Malcev Bases and Coordinates} On a technical level the
choice of the appropriate coordinate system for $G$ will be
essential. We will construct a particular set of strong Malcev
coordinates~\cite[Thm.~1.1.13]{CorwinGreenleaf}.

	\begin{lem} \label{StrongMB}
Let $\Lie{g}$ be a nilpotent Lie algebra of dimension $\dimG$ and let
$\Lie{g}_1 \subgr \ldots \subgr  \Lie{g}_l \subgr \Lie{g}$ be ideals with
$\dim(\Lie{g}_j) = m_j$. Then there exists a basis $\{ X_1, \ldots,
X_\dimG \}$ such that 
	\begin{itemize}
		\item[(i)] for each $1 \leq m \leq \dimG$, $\Lie{h}_m := \Rspan{X_1, \ldots, X_m}$ is an ideal of $\Lie{g}$;
		\item[(ii)] for $1 \leq j \leq l$, $\Lie{h}_{m_j} = \Lie{g}_j$.
	\end{itemize}
A basis satisfying (i) and (ii)  is called a strong Malcev basis
of $\Lie{g}$ passing through the ideals $\Lie{g}_1, \ldots, \Lie{g}_l $. 
	\end{lem}

 If we do not list any particular ideals, we will simply refer to the
 basis as a  strong Malcev basis.

	\begin{dfn}
Given a strong Malcev basis $\{ X_1, \ldots, X_\dimG \}$, the strong
Malcev coordinates of $G$ are defined by the map
	\begin{align}
		&\phi: \R^\dimG \to G, \nn \\
		&\phi(t_1, \ldots, t_\dimG) := \exp(t_1 X_1) \ldots \exp(t_\dimG X_\dimG) = \exp(t_1 X_1 * \ldots * t_\dimG X_\dimG). \label{DefStrMalcevCoord}
	\end{align}
	\end{dfn}

The Haar measure $\mu _G$ of a set $E\subseteq G$ is given by $\mu _G(E) = m(\phi \inv
(E))$, where $m$ is the Lebesgue measure on $\R^\dimG$. In Section~4 we
will fix a  Malcev basis and the corresponding Haar
measure. To determine the precise value of constants, we will
therefore keep track of all basis changes. 

The following two properties of strong Malcev coordinates~\cite[Prop.~1.2.7]{CorwinGreenleaf} will be used
later.

	\begin{lem} \label{LemPropStrongMC}
Let $\{ X_1, \ldots, X_\dimG \}$ be a strong Malcev basis for
$G$. Then there exist polynomial functions $P_{X_1}, \ldots, P_{X_\dimG}:
\R^{2 \dimG} \to \R$ of degree $\leq 2 \dimG$ such that for $t = (t_1,
\ldots, t_\dimG), s = (s_1, \ldots, s_\dimG) \in \R^\dimG$ the group
multiplication is given by 
	\begin{align}
		\phi(t) \phi(s) = \phi(P_{X_1}(t,s), \ldots, P_{X_\dimG}(t,s))\, . \label{PolyStrongMalcevMult}
	\end{align}
The polynomials have the following properties:
	\begin{itemize}
		\item[(i)] $P_{X_\dimG}(t, s) = t_\dimG + s_\dimG$ and $P_{X_{\dimG-1}}(t, s) = t_{\dimG-1} + s_{\dimG-1}$.
		\item[(ii)] If $1 \leq j \leq \dimG-2$, then $P_{X_j}(t, s) = t_j + s_j + \tilde{P}_{X_j}(t_{j+1}, \ldots, t_\dimG, s_{j+1}, \ldots, s_\dimG)$, where $2 \leq \deg(\tilde{P}_{X_j}) \leq (\dimG-j)^2$.
                \item[(iii)] $P_{X_j}(t, -t) = 0$ for $j=1, \dots, \dimG$. 
	\end{itemize}
\end{lem}
Given a strong Malcev basis, there exist polynomial functions $R_{X_1}, \ldots, R_{X_\dimG}: \R^\dimG \to \R$ such that
\begin{equation}
  \label{eq:p1}
		\phi(t) = \exp\Bigl(\sum_{j=1}^\dimG R_{X_j}(t) X_j \Bigr)  
\end{equation}
with similar properties as above:
	\begin{itemize}
		\item[(i)] $R_{X_\dimG}(t) = t_\dimG$ and $R_{X_{\dimG-1}}(t) = t_{\dimG-1}$,
		\item[(ii)] $R_{X_j}(t) = t_j+ \tilde{R}_{X_j}(t_{i+1}, \ldots, t_\dimG)$, where $2 \leq \deg(\tilde{R}_j) \leq \dimG$.
	\end{itemize}
The $R_{X_j}$ are the exponential coordinates of $g= \phi (t)$. 

For detailed proofs see \cite[Prop.~1.2.7]{CorwinGreenleaf}.

\subsection{Rational Structures, Uniform Subgroups,  and Quasi-Lattices}

Strong Malcev bases allow for the construction of discrete subsets of nilpotent groups which provide a convenient disjoint partition of $G$. Under certain conditions on the Lie bracket these discrete subsets are even subgroups. We recall the following definitions and statements from \cite[\S~ 5.1]{CorwinGreenleaf}.

	\begin{dfn}
Let $\Lie{g}$ be an $\dimG$-dimensional Lie algebra. We say $\Lie{g}$ has a rational structure if there is an $\R$-basis $\{ X_1, \ldots, X_\dimG \}$ for $\Lie{g}$ with rational structure constants. In particular, $\Lie{g} = \Lie{g}_\Q \otimes \R$ for $\Lie{g}_\Q := \Qspan{X_1, \ldots, X_\dimG}$.
	\end{dfn}

	\begin{dfn}
Let $G$ be a nilpotent Lie group and $\Gamma$ a discrete subgroup. We say that $\Gamma$ is a uniform subgroup if one of the following equivalent conditions holds:
	\begin{itemize}
		\item[(i)] $G/\Gamma$ is compact. In this case there
                  exists a relatively compact fundamental domain
                  $\Sigma  \subset G$ such that $\Sigma \Gamma = G$.
		\item[(ii)] There exists a strong Malcev basis $\{ X_1, \ldots, X_\dimG \}$ of $\Lie{g}$ such that
	\begin{align*}
		\Gamma = \exp(\Z X_1) \cdots \exp(\Z X_\dimG). 
	\end{align*}
	\end{itemize}
In this case we say $\{ X_1, \ldots, X_\dimG \}$ is strongly based on $\Gamma$.
	\end{dfn}

	\begin{thm}
Let $G$ be nilpotent with Lie algebra $\Lie{g}$.
	\begin{itemize}
		\item[(i)] If $G$ has a uniform subgroup $\Gamma$, then $\Lie{g}$ has a rational structure such that $\Lie{g}_\Q = \Qspan{\log(\Gamma)}$.
		\item[(ii)] Conversely, if $\Lie{g}$ has a rational structure from a $\Q$-algebra $\Lie{g}_\Q \subgr \Lie{g}$, then $G$ has a uniform subgroup $\Gamma$ such that $\log(\Gamma) \subseteq \Lie{g}_\Q$.
	\end{itemize}
	\end{thm}

	\begin{rem}
If $G$ possesses a uniform subgroup strongly based on the given Malcev
basis $\{ X_1, \ldots, X_\dimG \}$, then the  polynomial functions from Lemma~\ref{LemPropStrongMC} have the
additional property that
\begin{center}
 (iii) the $P_{X_j}$ and $R_{X_j}$ have rational coefficients.  
\end{center}
See \cite[Thm.~5.4.2]{CorwinGreenleaf}.
	\end{rem}

A uniform $\Gamma$ as in (ii) can be obtained by rescaling any strong Malcev basis with $\Lie{g}_\Q := \Qspan{X_1, \ldots, X_\dimG}$ by a sufficiently large integer.

In the absence of a rational structure we can still work with quasi-lattices.

	\begin{dfn}
 A discrete subset $\Gamma$  of $G$ is called a 
quasi-lattice with relatively compact fundamental domain $\Sigma$ if
$G = \bigcup_{\gamma \in \Gamma} \Sigma \gamma$ is a disjoint union of
translates. 
	\end{dfn}

Quasi-lattices always exist~\cite{FGr},  and an explicit
construction can be obtained by adapting the proof of Theorem~5.3.1 in
\cite{CorwinGreenleaf}. 

	\begin{lem} \label{LemmaQL}
Let $\{ X_1, \ldots, X_\dimG \}$ be a strong Malcev basis of $\Lie{g}$. Then
	\begin{align*}
		\Gamma = \bigl\{ \exp(k_1 X_1) \cdots \exp(k_\dimG X_\dimG) \mid k_1, \ldots, k_\dimG \in \Z \bigr\}
	\end{align*}
is a quasi-lattice in $G$ with fundamental domain
	\begin{align*}
		\Sigma = \bigl\{ \exp(t_\dimG X_\dimG) \cdots \exp(t_1 X_1) \mid t_1, \ldots, t_\dimG \in [0, 1) \bigr\}.
	\end{align*}
	\end{lem}

	\begin{proof}
\emph{Existence:} Let us set $\gamma_j := \exp(X_j)$. We prove by induction
on $\dim(G) = \dimG$ that for every $g \in G$ there exist  $k_1, \ldots, k_\dimG \in \Z$ and some $t_1, \ldots, t_\dimG \in [0, 1)$ such that
	\begin{align*}
		g = \exp(x_1 X_1) \cdots \exp(x_\dimG X_\dimG) =  \exp(t_\dimG X_\dimG) \cdots \exp(t_1 X_1) \, \gamma^{k_1}_1 \cdots \gamma^{k_\dimG}_\dimG.
	\end{align*}
The base case $\dimG = 1$ is trivial.

Induction step $\dimG \mapsto \dimG + 1$: For general $\dimG = \dim(G)
- 1$ we denote by $\Lie{g}_\dimG$ the ideal $\Rspan{X_1, \ldots,
  X_\dimG}$ and set $G_\dimG := \exp(\Lie{g}_\dimG)$. We know that for
$g \in G$ we have
$$g = \exp(x_1 X_1) \cdots \exp(x_{\dimG + 1} X_{\dimG + 1}) = g_\dimG
\, \exp(x_{\dimG + 1} X_{\dimG + 1})$$ with $g_\dimG \in G_\dimG$. We
write $x_{\dimG + 1} = t_{\dimG + 1} + k_{\dimG + 1}$ for some
uniquely determined $t_{\dimG + 1} \in [0, 1)$ and $k_{\dimG + 1} \in
\Z$. A short computation yields 
	\begin{align*}
		g &= g_\dimG \, \exp(t_{\dimG + 1} X_{\dimG + 1}) \exp(k_{\dimG + 1} X_{\dimG + 1}) \\
		&= \exp(t_{\dimG + 1} X_{\dimG + 1}) \, \Bigl( \exp(-t_{\dimG + 1} X_{\dimG + 1}) \, g_\dimG \, \exp(t_{\dimG + 1} X_{\dimG + 1}) \Bigr) \exp(k_{\dimG + 1} X_{\dimG + 1}) \\
		&=  \exp(t_{\dimG + 1} X_{\dimG + 1}) \, g'_{\dimG } \, \exp(k_{\dimG + 1} X_{\dimG + 1})
	\end{align*}
for some $g'_{\dimG } \in G_\dimG \nsubgr G$. By applying the induction hypothesis to $g'_{\dimG }$, we then obtain
	\begin{align*}
		g = \exp(t_{\dimG + 1} X_{\dimG + 1}) \, \exp(t_\dimG X_\dimG) \cdots \exp(t_1 X_1) \, \exp(k_1 X_1) \cdots \exp(k_\dimG X_\dimG) \, \exp(k_{\dimG + 1} X_{\dimG + 1}).
	\end{align*}
\emph{Uniqueness:} Assume that
	\begin{align*}
		\exp(t_\dimG X_\dimG) \cdots \exp(t_1 X_1) \, &\exp(k_1 X_1) \cdots \exp(k_\dimG X_\dimG) \\
		&= \exp(t'_\dimG X_\dimG) \cdots \exp(t'_1 X_1) \, \exp(k'_1 X_1) \cdots \exp(k'_\dimG X_\dimG) \\
		&= \exp(t_\dimG X_\dimG) \, g_{\dimG -1} \, \exp(k_\dimG X_\dimG) \\
		&= \exp(t'_\dimG X_\dimG) \, g'_{\dimG -1}\, \exp(k'_\dimG X_\dimG)
	\end{align*}
for some $g_{\dimG - 1}, g'_{\dimG - 1} \in G_{n-1}$. Then for $y_\dimG := t'_\dimG - t_\dimG$ we have
	\begin{align*}
		g_{\dimG - 1} &= \exp(y_\dimG X_\dimG) \, g'_{\dimG - 1} \, \exp\bigl( (k'_\dimG - k_\dimG) X_\dimG \bigr) \\
		&= \exp(y_\dimG X_\dimG) \, g'_{\dimG - 1} \, \exp(-y_\dimG X_\dimG) \, \exp\bigl( (y_\dimG + k'_\dimG - k_\dimG) X_\dimG \bigr) \\
		&= g''_{\dimG - 1} \, \exp\bigl( (y_\dimG + k'_\dimG - k_\dimG) X_\dimG \bigr)
	\end{align*}
for some $g''_{\dimG - 1} \in G_{n - 1}$. So ${g''_{\dimG -
    1}}^{-1} \, g_{\dimG - 1} \in \exp(\R X_\dimG)$, but $G_{\dimG -
  1} \bigcap \exp(\R X_\dimG) = \{ e \}$. We conclude that $g''_{\dimG -
  1} = g_{\dimG - 1}$ and $t'_\dimG - t_\dimG + k'_\dimG - k_\dimG =
0$, hence $t'_\dimG = t_\dimG$ and $k'_\dimG = k_\dimG$. By repeating
this argument for $\dimG - 1, \dimG - 2, \ldots, 1$, we obtain the 
uniqueness. 
	\end{proof}

	\begin{cor} \label{QLQuoGr}
Let $\{ X_1,
\ldots, X_\dimG \}$ be a strong Malcev basis strongly based on
$\Gamma$ passing through an ideal $\Lie{n} \nsubgr \Lie{g}$, and let
 $\Gamma$ be a quasi-lattice (or a uniform subgroup) of $G$ with
 fundamental domain $\Sigma $.   Let
 $N=\exp (\Lie{n})$ and  $\pr :
G \to N \rquo G$ denote  the natural map  onto the quotient $N\rquo  G$.
Then $\pr(\Gamma)$ is a quasi-lattice (a uniform 
subgroup)  of $N \rquo G$ with fundamental domain $\pr (\Sigma )$. 
	\end{cor}

	\begin{proof}
Suppose $\dim(N) = \dimG' <  n$, i.e., $\Lie{n} = \Rspan{X_1, \ldots,
  X_{\dimG '}}$ (cf.~Definition~\ref{StrongMB}). If $\Gamma$ is a
quasi-lattice, then it is clear from the proof of Lemma~\ref{LemmaQL}
that 
	\begin{align*}
		\pr(\Gamma) = \pr \bigl (\exp(k_{n'+1} X_{n'+1}) \cdots \exp(k_\dimG X_\dimG) \bigr)
	\end{align*}
is a quasi-lattice of $N \rquo G$ with fundamental domain
	\begin{align*}
		\pr(\Sigma) = \bigl\{ \exp(t_\dimG X_\dimG) \cdots
                \exp(t_{\dimG' +1} X_{\dimG' +1}) \mid t_1, \ldots, t_\dimG
                \in [0, 1) \bigr\}. 
	\end{align*}
The same reasoning applies if $\Gamma$ is a uniform subgroup of $G$,
since the proof of Lemma~\ref{LemmaQL} is an adaption of the proof of
\cite[Thm.~5.3.1]{CorwinGreenleaf}.   
(Alternatively, it is mentioned in \cite[Thm.~5.3.1]{CorwinGreenleaf} that $\Gamma \cap
G_{\dimG'}$ is uniform in $G_{\dimG'}$; so by \cite[Lem.~5.1.4~(a)]{CorwinGreenleaf},
 $\pr(\Gamma)$ is uniform in $G_{\dimG'} \rquo G$.) 
	\end{proof}

\subsection{Kirillov Theory  of Irreducible Representations} \label{SubsStandardModel}

According to Kirillov, every irreducible unitary representation $\pi $ of $G$ is constructed
by the following recipe: Choose $l\in \Lie{g}^*$ and choose a
subalgebra $\pid$ of maximal dimension, a so-called polarization,  such that 
$l([Y_1,Y_2]) = 0$ for all
$Y_1,Y_2\in \Lie{m}$. Consequently, $\chi _l:\PID  \to \mathbb{T} $ defined by 
$$
\chi _l(m) = e^{2\pi i \langle l , \log m\rangle }  \qquad \text{ for }
  m  \in \PID = \exp (\pid) \, 
$$
 is a character of $M$. The representation associated to $l$ 
 is the representation $\pi _l = \indR{\chi_l}{\PID}{G}$
induced from $\chi _l$ to $G$. It acts on the
representation space $\L{2}{\PID \rquo G} \cong
\L{2}{\R^d}$. It is well known that $(a)$  different polarizations of
$l$ yield  equivalent  representations, and $(b)$ different functionals
$l$ and $l' \in \Lie{g}^*$ in the same coadjoint orbit $\cO _l =
\coAd(G)l$ yield equivalent representations. Thus, every
orbit corresponds to a unique equivalence class of irreducible unitary
representations. In particular, we are free to choose the most
convenient functional $l$ and polarization $\pid$ when we write
a representation explicitly.

If a polarization $\pid $ is an ideal of $\Lie{g}$,  $\pid \nsubgr
\Lie{g}$, the induced representation $\pi _l = \indR{\chi_l}{\PID}{G}$ possesses the following
explicit description. Let $\QA = \PID \rquo G $ be the quotient group,  $q: G
\mapsto \PID \rquo G$ be the quotient
map, and  choose a
Borel cross-section $\sigma: \PID \rquo G \to G$ such that $q
\circ \sigma = \id$. Then every element $g\in G$ can be written in a
unique way as 
$$
g = p(g) \sigma (q(g)) = m h$$
with $p(g) = g \sigma (q(g))\inv \in \PID$. 
For $h' \in \sigma(\PID \backslash
G)$ we write 
	\begin{align} \label{ab}
		h'g = h' \hspace{2pt} m \hspace{2pt}h'^{-1}
                \hspace{2pt} h' h = h' \hspace{2pt} m \hspace{2pt}
                h'^{-1} \hspace{2pt} p(h' \hspace{2pt} h) \hspace{2pt}
                \sigma (q(h' \hspace{2pt} h)), 
	\end{align}
and  $h' \hspace{2pt} m \hspace{2pt} h'^{-1} \in \PID \nsubgr
G$. Then the  induced representation acting on  $f \in \L{2}{\PID \rquo G}$
is given by 
	\begin{align*}
		\Bigl(\indR{\chi_l}{\PID}{G}(g)f\Bigr)(q(h')) &= e^{2
                  \pi i \bracket{l}{\log\bigl(h' m h'^{-1} p(h' h)\bigr)}} \, f\bigl(q(h'
                \hspace{2pt} h)\bigr) \\ 
		&= e^{2 \pi i \bracket{l}{\log\bigl(h' m h'^{-1}\bigr)}} \, e^{2 \pi i \bracket{l}{\log\bigl(p(h' h)\bigr)}} \, f\bigl(q(h' \hspace{2pt} h)\bigr),
	\end{align*}
where the second equality holds because $\chi _l$ is a character on
$M$.   Let us point out that
	\begin{align*}
		\bracket{l}{\log(h' \hspace{2pt}m \hspace{2pt} h'^{-1})} = \bracket{\coAd(h'^{-1})l}{\log(m)}.
	\end{align*}
This observation explains why a  convenient description of the
coadjoint orbit $\Orbit_l$  helps to understand 
$\pi _l$. We will elaborate further in Subsection~\ref{ChRSubs}.

\subsection{Square-Integrable Representations}

An irreducible unitary representation $\pi $ of $G$ is square-integrable modulo the center of $G$ if there exists a vector $v \in \RS $ such that
 	\begin{align*}
		\int _{G/Z} \Abs{\bracket{\pi (\dot{g})v}{v}}^2 \, d\dot{g} < \infty.
	\end{align*}
We write $\pi \in SI/Z(G)$ when $\pi $ is  square-integrable modulo the center of $G$. A fundamental theorem of
Moore and Wolf~\cite{MooreWolf} characterizes all irreducible
representations in $SI/Z(G)$ of a nilpotent group by the
flat-orbit condition.

\begin{thm}[\cite{MooreWolf}] \label{ThmMW}
A representation $\pi _l$  of a connected, simply connected nilpotent group $G$ is
in $SI/Z(G)$ if and only if the coadjoint orbit $\Orbit _l = \mathrm{Ad}(G)^*l$  is flat, if and
only if $\Orbit_l \oplus \Liez{g}^* = \Lie{g}^*$, if and only if
$\Orbit _l = l +
\Lie{z}(\Lie{g})^\perp $, if and only if the symplectic form $B_l(\, . \,, \, . \,) := l([\, . \,, \, . \,])$ is non-degenerate on $\Lie{g} / \Liez{g}$.
\end{thm}

For any such $\pi_l$ one can restrict to representatives $l \in
\Liez{g}^*$. Furthermore, if $SI/Z(G) \neq \emptyset$, then up to a
set of Plancherel  measure zero all $\pi \in \widehat{G}$ are
square-integrable modulo $Z$. For each $\pi \in SI/Z(G)$ there exists
a number $\fd \in \R\setminus \{0\}$, called the formal degree or formal
dimension of $\pi$, such that 
	\begin{align}
		\fd \int_{G/Z(G)} \bracket{\pi(\dot{g})v_1}{w_1}
                \overline{\bracket{\pi(\dot{g})v_2}{w_2}} \,d\dot{g} =
                \bracket{v_1}{v_2}
                \overline{\bracket{w_1}{w_2}} \label{fd} 
	\end{align}
holds true for all $v_1, v_2, w_1, w_2 \in \RS$.

The formal degree depends on the normalization of the Haar measure $\mu$ on $G$ and is intimately connected to the symplectic form $B_l$: Pick a strong Malcev basis $\{Z_1, \ldots Z_r, X_1, \ldots, X_{2d}\}$ passing through $\Liez{g}$. Normalize the Lebesgue measure $\nu$ on $\Lie{g}$ such that
		\begin{equation} \label{NormHaar}
	\left\{ \begin{array}{rcl}
		\nu \bigl(\{ z_1 Z_1 + \ldots + x_{2d} X_{2d} \mid z_1, \ldots, x_{2d} \in [0,1] \} \bigr) &=& 1, \\
		\nu_{\Liez{g}} \bigl(\{ z_1 Z_1 + \ldots + z_r Z_r \mid z_1, \ldots, z_r \in [0,1] \} \bigr) &=& 1
	\end{array}\right.
	\end{equation}
and denote by $\mu$, $\mu_G$ or $dg$ the Haar measure on $G$ arising
from $\nu$ via strong Malcev coordinates. Since $Z(G)$ is normal in
$G$, we know that  $dg = dz \hspace{2pt} d\dot{g}$. We  fix this Haar
measure for once and all. Due to Theorem~\ref{ThmMW}, a 
square-integrable representation $\pi \in SI/Z(G)$ and its orbit 
$\Orbit_\pi$ depend only  on the representative $l \in
\Liez{g}^*$. So, the Pfaffian polynomial $\Pf$ defined by 
	\begin{align*}
		\Pf(l)^2 := \det\Bigl(B_l(X_j, X_k)_{j, k = 1}^{2d} \Bigr)
	\end{align*}
can be considered as a polynomial of $\Liez{g}^*$. By \cite[Thm.~4]{MooreWolf}, $\fd$ and $\Pf$ depend on the
normalization of $\mu$ in the same way; in fact we have 
	\begin{align} \label{pff} 
		\fd = \hspace{2pt} \Pf(l).
	\end{align}

	\begin{ex}
If $G$ is the Heisenberg group $\H$ and $\pi = \pi_\l$, the
Schr\"{o}dinger representation with parameter $\l \neq 0$, then $\fd =
\Abs{\l}^d$. (Cf.~Example~\ref{MainThmEuclVersion}.) 
	\end{ex}

\subsection{Graded Groups} \label{GrGrSubs}

	\begin{dfn}
A Lie algebra $\Lie{g}$ is graded if it possesses a vector space
decomposition $\Lie{g} = \bigoplus_{k = 1}^\infty \Lie{g}_k$ such that
 $[\Lie{g}_k, \Lie{g}_{k'}] \subseteq \Lie{g}_{k + k'}$ and 
 $\Lie{g}_k \neq \{ 0\}$ for only finitely many $k$. 
	\end{dfn}

We will write $\Lie{g}_N$ for the summand of highest order, whence
$\Lie{g}_k = \{ 0 \}$ for all $k > N$. However, some $\Lie{g}_k$ with $k
< N$ may be trivial, too. 

Every graded Lie algebra is nilpotent and the corresponding
connected, simply connected Lie group is called graded and is always
nilpotent locally compact.



	\begin{rem}
 Clearly $(\R^\dimG, +)$ and the Heisenberg group $\H$ are graded; $\H$ admits various gradations, one of which is given by
	\begin{align*}
		\h = {\h}_1 \oplus {\h}_2, \hspace{10pt} \mbox{ with } \hspace{10pt} {\h}_1 = \Rspan{Y_1, \ldots, X_d} \hspace{10pt} \mbox{ and } \hspace{10pt} {\h}_2 = \R Z.
	\end{align*}
Every nilpotent Lie algebra of  $\dim(\Lie{g}) \leq 6$  is graded.
The lowest dimension for which there are non-graded nilpotent Lie
algebras is $\dim(\Lie{g}) = 7$. 
	\end{rem}





A graded Lie algebra is naturally equipped with a  dilation.

	\begin{dfn}
Let $\Lie{g}$ be graded Lie algebra with gradation $\Lie{g} =
\bigoplus_{k = 1}^N \Lie{g}_k$. We define the family of dilations $\{
\dil_s \}_{s >0}$  on $\Lie{g}$ by
	\begin{align} \label{dilgrad}
		&\dil_s(X) := s^k \hspace{5pt} \mbox{ if } \hspace{5pt} X \in \Lie{g}_k.
	\end{align}
	\end{dfn}

Since we will deal with both gradations and rational structures at the
same time, we need some  compatibility of these structures. 

	\begin{dfn} \label{Compatibility}
Let $\Lie{g}$ be a graded Lie algebra with gradation $\Lie{g} = \bigoplus_{k = 1}^N  \Lie{g}_k$. Furthermore, suppose there exists a rational structure $\Lie{g}_\Q \subgr \Lie{g}$. We will say the rational structure is compatible with the gradation if there exists a strong Malcev basis $\{ X_1, \ldots, X_n \}$ of $\Lie{g}$ such that
	\begin{itemize}
		\item[(i)] $\Lie{g}_\Q := \Qspan{X_1, \ldots, X_n}$,
		\item[(ii)] $\{ X_1, \ldots, X_n \}$ passes through the ideals $\bigoplus_{k = l}^N  \Lie{g}_k \nsubgr \Lie{g}$, $l = N, \ldots, 1$.
	\end{itemize}
	\end{dfn}


\section{Constructing Orthonormal Bases} \label{Construction}

In this section we construct orthonormal bases in the orbit  of every
square-integrable representation of a graded
$SI/Z$-group. Our main theorems are the following.

	\begin{thm} \label{MainThmQL} [Quasi-Lattice]
Let $G$ be a graded $SI/Z$-group of $\dim(G) = 2d + 1$ with
$1$-dimensional center. Then there exists a normal subgroup $\PID
\nsubgr G$ with the following properties: 
	\begin{itemize}
		\item[(i)] The ideal $\pid = \log(\PID) \nsubgr
                  \Lie{g}$ is a polarization simultaneously for all $l \in
                  \Liez{g}^*$. It thus induces all $\pi \in SI/Z(G)$
                  and permits a universal realization of all $\pi \in
                  SI/Z(G)$ in $\L{2}{\PID \rquo G}$ related to  the
                  same set of strong Malcev coordinates for $G$. 
		\item[(ii)] For every $\pi \in SI/Z(G)$ with formal
                  degree $d_\pi $  there exist a
                  discrete subset $\Gamma \subseteq G/Z(G)$ and a
                  relatively compact set $\PFD \subseteq \PID \rquo
                  G$, such that 
	\begin{align} \label{eq:ll0}
		\bigl\{ \mu _{M \rquo G}(F)^{-1/2}  \hspace{2pt} \pi(\gamma) \, 1_{\PFD} \mid \gamma \in \Gamma \bigr\}
	\end{align}
forms an orthonormal basis of $\L{2}{\PID \rquo G}$. Furthermore,
$\Norm{\L{2}{\PID \rquo G}}{1_\PFD} = \fd^{1/2}$. 
	\end{itemize}
	\end{thm}
The construction shows that  $\Gamma $ is derived from a quasi-lattice
of $G/Z(G)$ by a small modification.  

	\begin{thm} \label{MainThmUnif} [Uniform Subgroup]
Let $G$ be a graded $SI/Z$-group of $\dim(G) = 2d + 1$ with $1$-dimensional center and a rational structure $\Lie{g}_\Q$ compatible with the gradation.
Then there exists a normal subgroup $\PID \nsubgr G$ with the following properties:
	\begin{itemize}
		\item[(i)] The ideal $\pid = \log(\PID) \nsubgr
                  \Lie{g}$ is a polarization simultaneously for all $l \in
                  \Liez{g}^*$. It thus induces all $\pi \in SI/Z(G)$
                  and permits a universal realization of all $\pi \in
                  SI/Z(G)$ in $\L{2}{\PID \rquo G}$ related to  the
                  same set of strong Malcev coordinates for $G$. 
		\item[(ii)] For every $\pi \in SI/Z(G)$
                  there exist a uniform subgroup $\Gamma \subgr
                  G/Z(G)$,  a relatively compact set $\PFD
                  \subseteq \PID \rquo G$,  such that 
	\begin{align*}
		\bigl\{  \mu _{M\rquo G}(F)^{-1/2} \hspace{2pt} \pi(\gamma) \, 1_\PFD \mid \gamma \in \Gamma \bigr\}
	\end{align*}
forms an orthonormal basis of $\L{2}{\PID \rquo G}$. Furthermore
$\Norm{\L{2}{\PID \rquo G}}{1_\PFD} = C^{-1} \hspace{2pt} \fd^{1/2}$
for some integer $C$ depending on $\Gamma $. 
	\end{itemize}
	\end{thm}

We will prove a
more general, but more technical statement (Theorem~\ref{TVMainThm})
and then  derive  Theorems~\ref{MainThmQL} and \ref{MainThmUnif} as
special cases. We will structure the proofs as follows:
	\begin{itemize}
		\item[(i)] Formulate an explicit parametrization of
                  the coadjoint orbits by means of the theorem of
                  Chevalley-Rosenlicht and  construct a
                  quasi-lattice or a uniform subgroup of $G$.
		\item[(ii)] Prove the  orthogonality.
		\item[(iii)] Prove the  completeness of the set in~\eqref{eq:ll0}.
	\end{itemize}


 Throughout we make the following assumptions:
	\begin{itemize}
		\item[--] $G$ is a connected, simply connected nilpotent $SI/Z$-group of dimension $\dimG = r + 2d$ with $r$-dimensional center $Z(G)$.
		\item[--] $\pi $ is a square-integrable representation
                  of $G$ modulo its center. 
      		\item[--] The coadjoint orbit that determines $\pi $
                  has an element $l \in \Liez{g}^*$.  
                \item[--] There exists an ideal $\pid \nsubgr \Lie{g}$
                  which is a polarization for  $l \in \Liez{g}^*$. 
		\item[--] The representation is  realized  on $\L{2}{\PID \rquo
                    G}$, where $\PID := \exp(\pid)$. 
	\end{itemize}

We will denote by $\qa $ the complementary subspace such that $\Lie{g}
= \pid \oplus \qa$. Since $\Lie{m}$ is an ideal, $\qa $ is a nilpotent
Lie algebra modulo $\Lie{m}$. The natural quotient maps  will be denoted
by  $\pr: G \to G/Z(G)$, $q: G \to \PID \rquo G$. For $\pr(g)$ we also
write $\dot{g}$. Given a  cross-section $\sigma : G \rquo \PID \to G$, the
projection $p: G\to M$ occurring in  the induced
representation is given by  the map
$p: G \to M$,  $p(g) = g \sigma (q(g))\inv \in \PID$. 



\subsection{Explicit Parameterizations of Coadjoint Orbits} \label{ChRSubs}

Our first goal is to write the induced representation $\pi = \pi _l$
associated to $l\in \Lie{z}(\Lie{g})  $ as explicitly as possible. 
We recall that 
 $\pi$  can
be written as  
	\begin{align}
		\bigl( \pi(g)f \bigr) \bigl( q(h') \bigr) 
 = e^{2 \pi i \bracket{\coAd(h'^{-1})l}{\log(m)}} \, e^{2 \pi i
   \bracket{l}{\log\bigl(p(h' h)\bigr)}} \, f\bigl(q(h' \hspace{2pt}
 h)\bigr). \label{FormulaRep} 
	\end{align}
The occurrence of $
\bracket{\coAd(h'^{-1})l}{\log(m)}$ suggests to describe the coadjoint
orbit $\Orbit_l$ in more detail.  

A celebrated theorem by Chevalley and Rosenlicht provides useful
parametrizations of the orbits of unipotent group actions. We will
formulate a version of the theorem 
adapted to the coadjoint action. For a general version we refer to 
\cite{CorwinGreenleaf}, Theorem\,3.1.4. and Section~3.4, for related parametrizations in solvable groups to \cite{ACD09, ACD12}.

	\begin{thm} [Chevalley-Rosenlicht] \label{AdaptedChevalleyRosenlicht}
Let $G$ be a connected, simply connected nilpotent $SI/Z$-group of
dimension $\dimG = r + 2d$ with $(r+d)$-dimensional ideal $\pid
\nsubgr \Lie{g}$. Denote by $\cdot : \pid^* \times G \to \pid ^*$  the restriction
of the coadjoint  action to $\pid ^*$, written as a right group action, 
 	\begin{align*}
		\subbracket{\pid^*}{l \cdot g}{v} := \subbracket{\Lie{g}^*}{\coAd(g^{-1})l}{v} = \subbracket{\Lie{g}^*}{l}{\Ad(g)v} \hspace{10pt} \mbox{ for all } \hspace{10pt} l \in \pid^*, v \in \pid.
	\end{align*}
Then $\pid^* \cdot G$  is a unipotent action. 

If $\pid$ is a polarization for some $l \in \Liez{g}^*$, then the
orbit $l \cdot G$ coincides with the $d$-dimensional affine subspace
$l + \Liez{g}^\perp \subseteq \pid^*$ and has the following properties: 
	\begin{itemize}
		\item[(i)]  For a fixed strong Malcev basis $\{Z_1,
                  \ldots, Z_r, Y_1, \ldots, Y_d, X_1, \ldots, X_d \}$
                  of $\Lie{g}$ passing through $\Liez{g}$ and $\pid$
                  there exists a basis $\{\tilde{X}_d, \ldots,
                  \tilde{X}_1\}$ of $\qa$ such that the map $\phi: \R^d \to l \cdot G
$ defined by 
	\begin{align*}
		\Phi (t_1, \ldots, t_d) = l \cdot \exp(t_d \tilde{X}_d) \cdot \ldots \cdot \exp(t_1 \tilde{X}_1)
	\end{align*}
		is a diffeomorphism.
In particular, $l \cdot G = \mathrm{Ad}^* (G) l = l  \cdot \exp(\R \tilde{X}_d) \cdots \exp(\R \tilde{X}_1)$.
		\item[(ii)] The basis is chosen such that  the
                  infinitesimal action $ \pid^*\cdot \Lie{g} $
                  satisfies 
		\begin{equation} \label{ChoiceXTilde}
	\left\{ \begin{array}{rcl}
		l \cdot \tilde{X}_1 &=& Y^*_1 \, \mathrm{mod} \, \Rspan{Y^*_2, \ldots, Y^*_d}, \\
		& \vdots \\
		l \cdot \tilde{X}_j &=& Y^*_j \, \mathrm{mod} \, \Rspan{Y^*_{j+1}, \ldots, Y^*_d}, \\
		& \vdots \\
		l \cdot \tilde{X}_d &=& Y^*_d.
	\end{array}\right.
	\end{equation}
In particular, $l \cdot G = l + l \cdot  \Lie{g} = l + \Rspan{Y^*_1, \ldots, Y^*_d}$.
		\item[(iii)] The polynomials $Q_1, \ldots, Q_d$ defined by
	\begin{align*}
		l \cdot \exp(t_d \tilde{X}_d) \cdot \ldots \cdot \exp(t_1 \tilde{X}_1) = l +  \sum_{j = 1}^d Q_j(t_d, \ldots, t_1) Y^*_j
	\end{align*}
are of the form
	\begin{align*}
		Q_j(t_j, \ldots, t_1) = t_j + \mbox{ a polynomial } \tilde{Q}_j(t_{j-1}, \ldots, t_1)
	\end{align*}
for all $j = 1, \ldots, d$.
	\end{itemize}
	\end{thm}

	\begin{proof}
Since $l \to l\cdot g$ is the restriction of
$\mathrm{Ad}^*(g\inv ) l$ to $\pid ^*$ and $\pid $ is an ideal, $\pid
^*$ is an invariant subspace of $\mathrm{Ad}^*$, and the restriction
of a unipotent action remains unipotent.

%
%
To prove (i) --- (iii), we apply \cite[Thm.~3.1.4]{CorwinGreenleaf} to $\pid^* \cdot
G$.
Let $l \in \Liez{g}^*$ be such that  $\pid$ is a polarization for
$l$. Flatness of the orbit $l \cdot G$ follows from that fact that $l
\cdot G$ is the image under the projection $\Lie{g}^* \to \pid^* =
\Lie{g}^*/\pid ^\perp $ of
the coadjoint orbit $\coAd(G) l = l + \Liez{g}^\perp \subseteq
\Lie{g}^*$.  


  Recall that for any strong Malcev
basis $\{Z_1, \ldots, X_d \}$ of $\Lie{g}$ the sequence of subspaces
$\Rspan{Z^*_1, \ldots, X^*_d} \subgr \Rspan{Z^*_2, \ldots, X^*_d}
\subgr \ldots \subgr \Lie{g}^*$ forms a Jordan-H\"{o}lder flag for the
coadjoint action of $G$. This property is then inherited by $\pid^*
\cdot G$ if the Malcev basis passes through $\Lie{z}^*$ and $\pid^*$:
If we quotient the vector space $\Lie{g}^*$ and, a fortiori, $G \to
\Aut(\Lie{g}^*): g \mapsto \coAd(g^{-1})$ by the subspace $\qa^*$,
then $\Rspan{Z^*_1, \ldots, Y^*_d} \subgr \Rspan{Z^*_2, \ldots, Y^*_d}
\subgr \ldots \subgr \pid^*$ forms a Jordan-H\"{o}lder flag for
$\pid^* \cdot G$. So  \cite[Thm.~3.1.4]{CorwinGreenleaf} yields sufficiently many vectors $\tilde{X}_j \in
\Lie{g}$ such that their combined action generates $l \cdot G$. The particular choice of $\{ \tilde{X}_j \}_j$ in the proof of \cite[Thm.~3.1.4]{CorwinGreenleaf} yields (i) --- (iii) except for the fact that these vectors form a basis of $\qa$.

To show this, let $Y, v \in \pid, X \in \qa$. Then
$\Ad(\exp(X))v \in \pid \nsubgr \Lie{g}$ and 
	\begin{align*}
		\bracket{l}{\Ad\bigl(\exp(Y) \exp(X) \bigr) v} &= \bracket{l}{\Ad(\exp(Y)) \bigl(\Ad(\exp(X))v \bigr)} \\
		&= \bracket{l}{\Ad(\exp(X))v + [Y, \Ad(\exp(X))v] + \ldots} \\
		&= \bracket{l}{\Ad(\exp(X))v}
	\end{align*}
since $l$ vanishes on $[\pid, \pid]$. Thus, the contribution of $\PID$
to $l \cdot G$ is trivial and we can choose the $\tilde{X}_j$ to be in $ \qa$. 
The linear independence of the $\tilde{X}_j$ becomes evident in view of the symplectic form $B_l$. The matrix representation of $B_l$ with respect to the basis $\{Z_1, \ldots, X_d \}$ is of the form
	\begin{align}
	[B_l]
	=
	\left[\begin{array}{c|ccc|ccc}
		l([\,.\,, \,.\,])&Y_1&\cdots &Y_d&X_1&\cdots&X_d \\ \hline
		Y_1&&&&*&\cdots&* \\
		\vdots&&\textnormal{\Large{0}}&&\vdots&\iddots&\vdots \\
		Y_d&&&&*&\cdots&* \\ \hline
		X_1&*&\cdots&*&*&\cdots&* \\
		\vdots&\vdots&\iddots&\vdots&\vdots&\iddots&\vdots \\
		X_d&*&\cdots&*&*&\cdots&* \\
	\end{array}\right]
		= 
	\left[\begin{array}{c|c}
		0&-A^* \\ \hline
		A&B \\
	\end{array}\right]. \label{B_l}
	\end{align}
Since $B_l$ is non-degenerate (Theorem~\ref{ThmMW}), the matrix $A$ is
invertible; so the infinitesimal action restricted to $\qa$ has full
rank and the vectors $\tilde{X}_d, \ldots, \tilde{X}_1$ must form a basis of $\qa$.
	\end{proof}


The following corollary is an immediate consequence of Theorem~\ref{AdaptedChevalleyRosenlicht} (ii).

	\begin{cor} \label{CorChR} 
The change of basis from $\{Z_1, \ldots, Z_r, Y_1, \ldots, Y_d, X_1, \ldots, X_d \}$ \newline to $\{Z_1, \ldots, Z_r, Y_1, \ldots, Y_d, \tilde{X}_d, \ldots, \tilde{X}_1 \}$ in Theorem~\ref{AdaptedChevalleyRosenlicht} (i) --- (ii) yields the matrix representation of the symplectic form

	\begin{align}
	\widetilde{[B_l]}
	=
	\left[\begin{array}{c|ccc|ccc}
		l([\,.\,, \,.\,])&Y_1&\cdots &Y_d&\tilde{X}_d&\cdots&\tilde{X}_1 \\ \hline
		Y_1&&&&&&-1 \\
		\vdots&&\textnormal{\Large{0}}&&&\iddots&* \\
		Y_d&&&&-1&*&* \\ \hline
		\tilde{X}_d&&&1&*&\cdots&* \\
		\vdots&&\iddots&*&\vdots&\iddots&\vdots \\
		\tilde{X}_1&1&*&*&*&\cdots&* \\
	\end{array}\right]
	=
	\left[\begin{array}{c|c}
		0&-\tilde{A}^* \\ \hline
		\tilde{A}&B \\
	\end{array}\right] . \label{tildeB_l}
	\end{align}

\qed
	\end{cor}

Note  that the vectors $\tilde{X}_d, \ldots, \tilde{X}_1$ are not
uniquely determined. According to \eqref{ChoiceXTilde} one has the liberty to add contributions of higher index in each step: $\tilde{X}_d$ is fixed by the choice of $l$ and
$Y_d$, whereas one can add some multiple of $\tilde{X}_d$ to
$\tilde{X}_{d-1}$ without changing the assertions of the theorem, and
so forth. The triangularization of  $A$ can be interpreted as the first
step of an algorithm which produces the coefficient polynomials $Q_j,
j =1, \ldots, d$, from Theorem~\ref{AdaptedChevalleyRosenlicht} (iii).

In what follows we would like to use the basis $\{ Z_1, \ldots, Y_d,
\tilde{X}_d, \ldots, \tilde{X}_1 \}$ to construct a
quasi-lattice or a uniform subgroup in $G$. According to Lemma~\ref{LemmaQL} we
need a Malcev basis, not just a vector space basis, for this
construction.  This motivates the following definition.

	\begin{dfn}
Let $G$, $\pid$, $\qa$ be as above and let $l \in \Liez{g}^*$. Let $\{
Z_1, \ldots, Y_d\}$ be  a strong Malcev basis  of $\pid$
passing through $\Liez{g}$ and let  $\{ \tilde{X}_d, \ldots,
\tilde{X}_1 \}$  be a basis of $\qa$  satisfying the conclusions (i)
and (ii) of Theorem~\ref{AdaptedChevalleyRosenlicht} with respect to $l$.
If the  union $\{ Z_1, \ldots, Y_d, \tilde{X}_d, \ldots, \tilde{X}_1
\}$ forms a strong Malcev basis of $\Lie{g}$,  it is  called 
a Chevalley-Rosenlicht-admissible (or Ch-R-admissible) Malcev basis of
$\Lie{g}$ subordinate to $l$. 
	\end{dfn}

In this paper Chevalley-Rosenlicht admissible Malcev bases are simply a
technical tool. At this time we do not know whether and when a
Ch-R-admissible basis exists in a nilpotent Lie algebra. In
Subsection~\ref{GrSIZGrSubs} we will see that every graded Lie group with
one-dimensional center admits a Ch-R-admissible basis. 

Since the basis of the complement $\Lie{h}$ in
Theorem~\ref{AdaptedChevalleyRosenlicht} is not unique, the first
question is whether the definition of Ch-R-admissible bases depends on
the choice of the basis. Fortunately this is not the case. 

\begin{lem} \label{lem-chrad}
Assume that $\ggl $ possesses a Ch-R-admissible Malcev basis
 subordinate to $l\in \Liez{g}^*$. Then every basis satisfying
 \eqref{ChoiceXTilde} is Ch-R-admissible.
\end{lem}

\begin{proof}
Let $\{ \tilde{X}_d, \ldots, \tilde{X}_1 \}$ and $\{ \tilde{\tilde{X}}_d, \ldots, \tilde{\tilde{X}}_1 \}$ be two bases of $\qa$ to parametrize the orbit $l \cdot G$ in $\pid^* \cdot G$. According to \eqref{tildeB_l} this means that the $d \times d$-matrix $\tilde{A}$ with entries $\tilde{A}_{j,k} := l([\tilde{X}_j, Y_k])$ is an upper triangular matrix with ones on the diagonal. Likewise, $\tilde{\tilde{A}}$, with $\tilde{\tilde{A}}_{j,k} := l([\tilde{\tilde{X}}_j, Y_k])$, is upper triangular. Consequently, $C := \tilde{\tilde{A}} \tilde{A}^{-1}$ is upper triangular. In fact, the matrix $C$ describes the basis change from $\{ \tilde{X}_j \}_{j=1}^d$ to $\{ \tilde{\tilde{X}}_j \}_{j=1}^d$, because for $\tilde{\tilde{X}}_j = \sum_{n \geq j} c_{j, n} \hspace{2pt} \tilde{X}_n$ we have
	\begin{align}
		l([\tilde{\tilde{X}}_j, Y_k]) &= \sum_{n \geq j} c_{j, n} \hspace{2pt} l([\tilde{X}_n, Y_k]) \hspace{15pt} \mbox{ or equivalently} \nn \\
		\tilde{\tilde{A}} &= C \tilde{A}. \label{MatrixC}
	\end{align}

Now, assume that $\{ Z_1, \ldots, Y_d, \tilde{X}_d, \ldots, \tilde{X}_1 \}$ is a strong Malcev basis for $\Lie{g}$. Then for $j < k$ we have $[\tilde{X}_j, \tilde{X}_k] \in \Rspan{Z_1, \ldots, Y_d, \tilde{X}_d, \ldots, \tilde{X}_k} := \Lie{g}_k$, or $[\tilde{X}_j, \tilde{X}_k] \in \Lie{g}_{\max (j, k)}$. Since $C$ is upper triangular, we find that $\tilde{\tilde{X}}_k = \sum_{m \geq k} c_{k, m} \hspace{2pt} \tilde{X}_m \in \Lie{g}_k$ and consequently, $\Lie{g}_k = \Rspan{Z_1, \ldots, Y_d, \tilde{\tilde{X}}_d, \ldots, \tilde{\tilde{X}}_k}$.
We check the Lie bracket for the basis $\{ \tilde{\tilde{X}}_j \}_{j=1}^d$. For $j <k$ we have
	\begin{align*}
		[\tilde{\tilde{X}}_j, \tilde{\tilde{X}}_k] = \sum_{n
                  \geq j} \sum_{m \geq k} c_{j, n} \hspace{2pt} c_{k,
                  m} \hspace{2pt} [\tilde{X}_n, \tilde{X}_m] \, ,
	\end{align*}
and therefore $[\tilde{\tilde{X}}_j, \tilde{\tilde{X}}_k] \in \Lie{g}_k$.
This implies that $\{ Z_1, \ldots, Y_d, \tilde{\tilde{X}}_d, \ldots, \tilde{\tilde{X}}_1 \}$ is a strong Malcev basis for $\Lie{g}$.

Clearly, the argument is symmetric. So, $\{ Z_1, \ldots, Y_d, \tilde{X}_d, \ldots, \tilde{X}_1 \}$ is Ch-R-admissible subordinate to $l$ if and only if $\{ Z_1, \ldots, Y_d, \tilde{\tilde{X}}_d, \ldots, \tilde{\tilde{X}}_1 \}$ is so.
\end{proof}

Once we know that $\Lie{g}$  possesses  a Ch-R-admissible basis subordinate to $l$, we have a certain freedom to choose. A distinguished Ch-R-admissible basis is obtained by choosing the basis transformation to be $C = \tilde{A}^{-1}$ in \eqref{MatrixC}. Then $\tilde{\tilde{A}} = C \tilde{A} = \tilde{A}^{-1} \tilde{A} = I$ or, in other words,
	\begin{align}
		\tilde{\tilde{A}}_{j, k} &= l([\tilde{\tilde{X}}_j, Y_k]) = \delta_{j, k}, \hspace{10pt} j, k = 1, \ldots, d, \hspace{10pt} \mbox{ or } \nn \\
		\tilde{\tilde{A}} &=
	\left[\begin{array}{ccc}
		\textnormal{\Large{0}}&&1 \\
		&\iddots& \\
		1&&\textnormal{\Large{0}}
	\end{array}\right]. \label{tildeAdist}
	\end{align}
Thus, if the action $\pid^* \times G \to \pid^*$ possesses a Ch-R-admissible basis subordinate to $l$, then it possesses a distinguished Ch-R-admissible basis that diagonalizes the infinitesimal coadjoint action on $l$.

\subsection{A Technical Version of the Main Theorem} \label{ProofMainThmSubs}
We formulate and prove a general, but technical version for the
existence of orthonormal bases associated to a representation $\pi \in
SI/Z(G)$. 

	\begin{thm} 
          \label{TVMainThm}
Let $G$ be a connected, simply connected nilpotent $SI/Z$-group of
dimension $\dimG = r + 2d$ with $(r+d)$-dimensional ideal $\pid
\nsubgr \Lie{g}$. Fix a strong Malcev basis \newline $\{Z_1, \ldots, Z_r, Y_1,
\ldots, Y_d, X_1, \ldots, X_d \}$ of $\Lie{g}$ which passes through
$\Liez{g}$ and $\pid$; use it to fix the Haar measures $\mu_{G}$ and
$\mu_{Z(G)}$. Let $\pi \in SI/Z(G)$ with representative $l \in
\Liez{g}^*$. 

(A) If 
	\begin{itemize}
		\item[(i)] $\pid$ is a polarization for $l$,
		\item[(ii)] $\Lie{g}$ has a Ch-R-admissible Malcev basis subordinate to $l$ of the form \newline $\{ Z_1, \ldots, Y_d, \tilde{X}_d, \ldots, \tilde{X}_1 \}$,
	\end{itemize}
then there exist a discrete subset $\Gamma \subseteq G/Z(G)$ and 
 a relatively compact subset $F \subseteq M \rquo G$  
such that 
the set 
	\begin{align*}
		\bigl\{ \mu _{M \rquo G}(F)^{-1/2} \hspace{2pt} \pi(\gamma) \, 1_{\PFD} \mid \gamma \in \Gamma \bigr\}
	\end{align*}
forms an orthonormal basis of $\L{2}{\PID \rquo G}$. 

(B) If furthermore
	\begin{itemize}
		\item[(iii)] $\Lie{g}_\Q := \Qspan{Z_1, \ldots, X_d}$ forms a rational subalgebra of $\Lie{g}$,
		\item[(iv)] $l \in {\Lie{z}(\Lie{g}_\Q)}^*$,
		\item[(v)] $\Lie{g}_\Q$ possesses a Ch-R-admissible basis $\{ Z_1, \ldots, Y_d, \tilde{X}_d, \ldots, \tilde{X}_1 \}$ subordinate to $l$,
	\end{itemize}
then there exist a uniform subgroup $\Gamma \subgr G/Z(G)$ of
co-volume $\Gmeas{G/Z(G)}{\Sigma} = \fd^{-1}$ with fundamental domain
$\Sigma $, such that for $\PFD := q(\Sigma) \subseteq \PID \rquo G$ the set 
	\begin{align} \label{saturd}
		\bigl\{  \mu _{\PID \rquo G }(F) ^{-1/2} \hspace{2pt} \pi(\gamma) \, 1_\PFD \mid \gamma \in \Gamma \bigr\}
	\end{align}
forms an orthonormal basis of $\L{2}{\PID \rquo G}$. Furthermore,
$\Norm{\L{2}{\PID \rquo G}}{1_\PFD} = C^{-1} \hspace{2pt} \fd^{1/2}$
for some integer $C\in \N $. 
	\end{thm}

The theorem will be  proved in three steps: First we  construct a
quasi-lattice or a uniform subgroup in $G/Z(G)$, then we  verify the
orthonormality of the set \eqref{saturd}, and finally we show its completeness.

	\begin{prop} \label{Subgroup} [Step 1. Construction of Quasi-Lattice/Uniform Subgroup.]

\noindent Assume that $G$ satisfies  the assumptions of
Theorem~\ref{TVMainThm}. 
	\begin{itemize}
		\item[(i)] If $G$ satisfies conditions (i) and (ii),
                  then  there exists a quasi-lattice $\Gamma'
                  \subseteq G/Z(G)$ of co-volume $\mu_{G/Z(G)}(\Sigma)
                  = \fd^{-1}$. 
		\item[(ii)] If   $G$ satisfies all  conditions (i) ---
                  (v), then there exists a uniform subgroup $\Gamma'
                  \subgr G/Z(G)$ of co-volume $\mu_{G/Z(G)}(\Sigma) =
                  \fd^{-1}$. 
	\end{itemize}
	\end{prop}

	\begin{proof}

(i) Let $\{ Z_1, \ldots, \tilde{X}_d \}$ be a Ch-R-admissible Malcev basis of $\Lie{g}$ subordinate to $l \in \Liez{g}^*$. By Lemma~\ref{LemmaQL}, the set
	\begin{align*}
		\Gamma'_G = \exp(\Z Z_1) \cdots \exp(\Z \tilde{X}_1)
	\end{align*}
forms a quasi-lattice of $G$ with fundamental domain
	\begin{align*}
		\Sigma_G := \{ \exp(\tilde{x}_{1} \tilde{X}_1) \cdots \exp(z_1 Z_1) \mid z_1, \ldots, \tilde{x}_{1} \in [0,1) \},
	\end{align*}
and by Corollary~\ref{QLQuoGr} $\Gamma' := \pr(\Gamma'_G)$
is a quasi-lattice of $G/Z(G)$ with fundamental domain $\Sigma :=
\pr(\Sigma _G)$.
Let  $S_l$ denote the matrix of the basis  change from $\{ Z_1, \ldots, Y_d,
X_1, \ldots, X_d \}$ to $\{ Z_1, \ldots, Y_d, \tilde{X}_d, \ldots,
\tilde{X}_1 \}$, then the pullback under $S_l$ of the quotient Haar
measure $\mu_G$ (fixed by \eqref{NormHaar}) is given by $\tilde{\mu}_G
= \det(S_l)^{-1} \hspace{2pt} \mu_G$~\cite[1.2.11]{CorwinGreenleaf}. 
We now compute the formal degree
$\widetilde{\fd}$ of $\pi $  with respect to
$\widetilde{\mu}_{G/Z(G)}$ in two ways. On the one hand, 
 $\widetilde d _\pi = \det(S_l) \hspace{2pt} \fd$ by \eqref{fd}. On the
 other hand, the relation \eqref{pff} between the formal degree and the
 Pfaffian  with
respect to the basis $\{ Z_1, \ldots, \tilde{X}_1 \}$ yields 
	\begin{align*}
		\widetilde{\fd}^2 = \widetilde{\Pf(l)}^2 = \det\bigl(
                \widetilde{[B_l]} \bigr) = \det(\tilde{A})^2 = 1 \, .
              	\end{align*}
We conclude that 
\begin{equation}
  \label{eq:satur1}
 \fd =     \det(S_l)^{-1}.
\end{equation}
Finally we compute the measure of the fundamental domain $\Sigma _G$
for $\Gamma _G'$. Since  $\Sigma_G$ is the image  of the unit cube
spanned by $\{Z_1, \ldots, \tilde{X}_1 \}$, the normalization
\eqref{NormHaar} implies that  $1 = \widetilde{\mu}_{G/Z(G)}(\Sigma) =
\det(S_l)^{-1} \hspace{2pt} \mu_{G/Z(G)}(\Sigma)$. Thus, we have 
	\begin{align}
		\mu_{G/Z(G)}(\Sigma) = \det(S_l) = \fd^{-1}.\nn 
	\end{align}


(ii) 
To construct a uniform subgroup $\Gamma ' \leq G/Z(G)$, we choose  a distinguished Ch-R-admissible basis 
$\{\tilde{X}_j \}_{j =1}^d$
which satisfies \eqref{tildeAdist}. By the proof  of Lemma~\ref{lem-chrad} this means that $C = A^{-1}$ and that
	\begin{equation} \label{eq:c56}
		\tilde X_j = \sum _{m=1}^d c_{j, m} X_{m} \, .  
	\end{equation}
Now assume that $\Lie{g}_\Q = \Qspan{Z_1,
   \ldots, Y_d, X_1, \ldots, X_d}$ and $l \in \Liez{\ggl _\Q }^*$. Then
 $l = \sum _{n=1}^r  l_n Z^*_n$ with coefficients $l_n \in
   \Q$ and 
\begin{align*}
		[X_j, Y_{k}] = \sum_{m = 1}^d   c^{j, k}_m Y_m + \sum
                _{n=1}^r  d^{j, k}_n Z_n , 
	\end{align*}
again with coefficients $c^{j, k}_m, d^{j, k}_n \in \Q$.
Consequently, $l([X_j,Y_k]) = \sum _{n=1}^r l_n d_n^{j, k}$ is in $\Q$
and thus the matrix $A$ has only rational entries. By Kramer's rule
$A\inv $ has only rational entries, and  \eqref{eq:c56}
implies that $\tilde X_j \in \ggl _\Q $. It follows that 
$\Lie{g}_\Q = \Qspan{Z_1,
   \ldots, Y_d, \tilde{X}_d, \ldots, \tilde{X}_1}$.

In order to construct a uniform subgroup $\Gamma'$, we temporarily
enumerate $\{Z_1, \ldots, \tilde{X}_1\}$ by $\{ V_1, \ldots, V_\dimG
\}$. Let $P_{V_1}, \ldots, P_{V_\dimG}$ be the 
polynomials from~\eqref{PolyStrongMalcevMult} that  describe the
multiplication in $G$ in strong Malcev coordinates. Furthermore, let $K \in \N$ be large enough so that all denominators of all coefficients of all $P_j$ divide $K$. We then set
	\begin{equation}\label{satur2}
		\left\{ \begin{array}{lcr}
			W_j := q_j V_j := K^{-2^{d+1}} V_j &\hspace{10pt} \mbox{ for } \hspace{10pt}&	j =1, \ldots, r, \\
			W_{r + j} := q_{r + j} V_{r + j} := K^{-2^{d+1-j}} V_{r + j} &\hspace{10pt} \mbox{ for } \hspace{10pt}&	j =1, \ldots, d, \\
			W_{r + d + j} := q_{r + d + j} V_{r + d + j} := K^{2^j} V_j &\hspace{10pt} \mbox{ for } \hspace{10pt}&	j =1, \ldots, d. \\
		\end{array} \right.
	\end{equation}
The rescaled basis is
	\begin{align*}
		\{ W_1, \ldots, W_\dimG \} = \{ K^{-2^{d+1}} Z_1, \ldots, K^{-2^{d+1}} Z_r, K^{-2^d} Y_1, \ldots, K^{-2} Y_d, K^2 \tilde{X}_d, \ldots, K^{2^d} \tilde{X}_1 \}.
	\end{align*}
If we denote by $P'_{V_j}$ the polynomials from \eqref{PolyStrongMalcevMult} for the strong Malcev basis $\{ W_1, \ldots, W_\dimG \}$, then, by definition of the $q_j$'s,
	\begin{align*}
		q_j P'_{V_j}(t_1, \ldots, t_\dimG, s_1, \ldots, s_\dimG) = P_{V_j}(q_1 t_1, \ldots, q_\dimG t_\dimG, q_1 s_1, \ldots, q_\dimG s_\dimG)
	\end{align*}
and $P'_{V_j}$ has integer coefficients for all $j = 1, \ldots, \dimG$. Hence,
	\begin{align*}
		\Gamma'_G := \exp\bigl(\Z W_1 \bigr) \cdots \exp\bigl(\Z W_\dimG \bigr)
	\end{align*}
forms a uniform subgroup of $G$ with fundamental domain
	\begin{align*}
		\Sigma_G := \{ \exp(w_n W_n) \cdots \exp(w_1 W_1) \mid w_1, \ldots, w_n \in [0,1) \}.
	\end{align*}
By Corollary~\ref{QLQuoGr}, $\Gamma' := \pr(\Gamma'_G)$ forms a
uniform subgroup of $G/Z(G)$ with fundamental domain $\Sigma :=
\pr(\Sigma_G)$. Since the Jacobian determinant of the rescaling in
$G/Z(G)$  equals $1$, the same reasoning as in (i) yields
$\mu_{G/Z(G)}(\Sigma) = \fd^{-1}$ for case (ii). This completes the
proof. 
	\end{proof}

	\begin{rem} \label{ReparametrizingChR}
Multiplying the basis vectors by the factors $q_j$ does not change the
assertion (ii) of the Chevalley-Rosenlicht theorem. The principal
terms are linear in $t = (t_d, \ldots, t_1)$ because they correspond
to the infinitesimal group action $\coad$. Indeed, our choice of
factors in \eqref{satur2} yields   
	\begin{align*}
		\bracket{\coad(-t_j K^{2^{d+1-j}} \tilde{X}_j) l}{K^{-2^{d+1-j}} Y_j} &= \bracket{l}{[-t_j K^{2^{d+1-j}} \tilde{X}_j, K^{-2^{d+1-j}} Y_j]} \\
		&= \bracket{\coad(t_j \tilde{X}_j) l}{Y_j} \\
		&= \bracket{\Bigl( t_j  Y^*_j + \sum _{m=j+1}^{d} t_m Y^*_m\Bigr)}{Y_j} = t_j.
	\end{align*}
	\end{rem}

	\begin{dfn} \label{DiscreteSet}
Let $\Gamma'$ be a quasi-lattice or a uniform subgroup as constructed 
in Proposition~\ref{Subgroup}. If $\Gamma '$ is a quasi-lattice, set
$K=1$; if $\Gamma '$ is a uniform subgroup, let $K$ be the scaling
factor in \eqref{satur2}.  Then we set 
	\begin{equation*}
	\begin{array}{lcl}
		\Gamma'_\PID& := &\exp(\Z  K^{-2^{d}} Y_1) \cdots \exp(\Z  K^{-2} Y_d), \\
		\Gamma'_\QA& := &\exp(\Z  K^2 \tilde{X}_d) \cdots \exp(\Z  K^{2^d} \tilde{X}_1), \\
		\Gamma& := &\pr \bigr( \Gamma_\PID \cdot
                \Gamma_\QA^{-1}  \bigr) \subseteq G/Z(G) \, .
	\end{array}
	\end{equation*}
Let $\Sigma $ be a fundamental domain for $\Gamma '$ in $G$ and set $F
= q(\Sigma ) \subset M \rquo G$. In coordinates, 
	\begin{align*}
		F =\Bigl\{ q \Bigl( \exp(t_1  K^{2^d} \tilde{X}_1) \dots \exp(t_d  K^{2} \tilde{X}_d) \Bigr) \mid t_1, \ldots, t_d \in [0, 1)^d \Bigr\}.
	\end{align*}
	\end{dfn}

	\begin{rem}
If $\Gamma '$ is a uniform subgroup, then  $\Gamma = \Gamma
'$. If $\Gamma '$ is only a quasi-lattice, then $\Gamma $ is obviously
related to $\Gamma '$, but in general $\Gamma \neq \Gamma '$. 

Let $\vartheta \in \Gamma _M'$ and $\eta \in \Gamma _H'$, then  every
element of $\Gamma$ is of the form  $\gamma =
\pr \bigr( \vartheta \hspace{2pt} \eta^{-1} \bigl) $ for
	\begin{align*}
		\vartheta \hspace{2pt} \eta^{-1} = \exp(\vartheta_1  K^{-2^d} Y_1) \cdots \exp(\vartheta_d  K^{-2} Y_d) \exp(-\eta_1  K^{2^d} \tilde{X}_1) \cdots \exp(-\eta_d  K^{2} \tilde{X}_d) \in \Gamma_\PID \cdot \Gamma_\QA^{-1}
	\end{align*}
with $\vartheta_1, \ldots, \eta_1 \in \Z$. 
	\end{rem}

We will need a more  explicit description of the operators
$\pi(\gamma)$, $\gamma \in \Gamma$,  in terms of given  strong
Malcev coordinates. Of particular importance will be a specific type
of products $g \cdot g' \in G$, for which we give a description in the
spirit of Lemma~\ref{LemPropStrongMC}. 

	\begin{lem} \label{PolyQL}
Let $\{ X_1, \ldots, X_\dimG \}$ be a strong Malcev basis of $\Lie{g}$. Then there exist polynomial functions $S_{X_j}, T_{X_k}: \R^{\dimG} \to \R$ of degree $\leq 2 \dimG$ for all $j, k = 1, \ldots, 2 \dimG$ such that for $t = (t_1, \ldots, t_\dimG), s = (s_1, \ldots, s_\dimG) \in \R^\dimG$ we have
	\begin{align*}
		\exp(t_\dimG X_\dimG) \cdots \exp(t_1 X_1) \cdot \exp(s_1 X_1) \cdots \exp(s_\dimG X_\dimG) &= \exp \bigl( S_{X_1}(t, s) X_1 \bigr) \cdots \exp \bigl( S_{X_\dimG}(t, s) X_\dimG \bigr), \\
		\exp(t_1 X_1) \cdots \exp(t_\dimG X_\dimG) \cdot \exp(s_\dimG X_\dimG) \cdots \exp(s_1 X_1) &= \exp \bigl( T_{X_1}(t, s) X_1 \bigr) \cdots \exp \bigl( T_{X_\dimG}(t, s) X_\dimG \bigr)
	\end{align*}
with
	\begin{itemize}
		\item[(i)] $S_{X_j}(t, s) = t_j + s_j + \tilde{S}_{X_j}(t_{j+1}, \ldots, t_\dimG, s_{j+1}, \ldots, s_\dimG)$,
		\item[(ii)] $T_{X_k}(t, s) = t_k + s_k + \tilde{T}_{X_k}(t_{k+1}, \ldots, t_\dimG, s_{j+1}, \ldots, s_\dimG)$.
	\end{itemize}

	\end{lem}

	\begin{proof}
 The proof 
resembles the proof of \cite[Proposition~1.2.7]{CorwinGreenleaf} and
uses induction on $\dimG$.  For $\dim(\Lie{g}) = \dimG = 1$ or $n=2$ the
result is trivial since $\Lie{g}$ is Abelian.

Thus, suppose $\dimG > 2$ and that (i) holds true for all $j = 2, \ldots, \dimG$. Let $\Lie{g} \to \Lie{g}/\Lie{g}_1: X \mapsto \overline{X}$ be the canonical projection modulo the central ideal $\Lie{g}_1 := \R X_1$. By induction,
	\begin{align*}
		\exp(t_\dimG \overline{X}_\dimG) \cdots \exp(t_2 \overline{X}_2) \cdot \exp(s_2 \overline{X}_2) \cdots \exp(s_\dimG \overline{X}_\dimG) &= \exp \bigl( S_{\overline{X}_2}(t, s) \overline{X}_2 \bigr) \cdots \exp \bigl( S_{\overline{X}_\dimG}(t, s) \overline{X}_\dimG \bigr)
	\end{align*}
with the $S_{\overline{X}_j}$ satisfying (i). For $j=2, \dots ,n $ we
set $S_{X_j} := S_{\overline{X}_j}$. Since $\Lie{g}_1$ is central,
there exists a polynomial function $\tilde{S}_{X_1}$ such that the
multiplication in $G$ is 
	\begin{align*}
		\exp(t_\dimG X_\dimG) \cdots& \exp(t_2 X_2) \cdot \exp(s_2 X_2) \cdots \exp(s_\dimG X_\dimG) \\
		&= \exp \bigl( \tilde{S}_{X_1}(t_2, \ldots, t_\dimG, s_2, \ldots, s_\dimG) X_1 \bigr) \cdot \exp \bigl( S_{X_2}(t, s) X_2 \bigr) \cdots \exp \bigl( S_{X_\dimG}(t, s) X_\dimG \bigr)
	\end{align*}
and, consequently,
	\begin{align*}
		\exp(t_\dimG& X_\dimG) \cdots \exp(t_2 X_2) \exp(t_1 X_1) \cdot \exp(s_1 X_1) \cdot \exp(s_2 X_2) \cdots \exp(s_\dimG X_\dimG) \\
		&= \exp \bigl( (t_1 + s_1) X_1 \bigr) \cdot \exp \bigl( \tilde{S}_{X_1}(t_2, \ldots, s_\dimG) X_1 \bigr) \cdot \exp \bigl( S_{X_2}(t, s) X_2 \bigr) \cdots \exp \bigl( S_{X_\dimG}(t, s) X_\dimG \bigr) \\
		&= \exp \bigl( (t_1 + s_1 + \tilde{S}_{X_1}(t_2, \ldots, s_\dimG)) X_1 \bigr) \cdot \exp \bigl( S_{X_2}(t, s) X_2 \bigr) \cdots \exp \bigl( S_{X_\dimG}(t, s) X_\dimG \bigr).
	\end{align*}
This proves (i). The proof of  (ii) works out analagously.
	\end{proof}

	\begin{prop} \label{Orthonormality} [Step 2.  Orthonormality]

\noindent Let $\Gamma$ be the discrete subset of  $G/Z(G)$ from
Definition~\ref{DiscreteSet}  and let $C := K^{2^d-1}\in \N$. Then the set
	\begin{align*}
		\bigl\{ \mu _{M \rquo G } (F)^{-1/2} \hspace{2pt} \pi(\gamma) \, 1_\PFD \mid \gamma \in \Gamma \bigr\}
	\end{align*}
forms an orthonormal system of $\L{2}{\PID \rquo G}$. Furthermore
$\Norm{\L{2}{\PID \rquo G}}{1_\PFD} = C^{-1} \hspace{2pt} \fd^{1/2}$. 
	\end{prop}

	\begin{proof}
To prove orthogonality, we have to show that for $\gamma \neq  \gamma' \in \Gamma$
$$
		\subbracket{\L{2}{\PID \rquo
                    G}}{\pi(\gamma)1_\PFD}{\pi(\gamma')1_\PFD} = 0 \,
                .
$$

Writing $\gamma = \pr(\vartheta \eta \inv) $, the explicit formula for the
representations  \eqref{FormulaRep} yields 
	\begin{align}
		\Big( \pi(\gamma)1_\PFD \Big)\big( q(h) \big) =  
                e^{2 \pi i
                  \bracket{\coAd(h^{-1})l}{\log(\vartheta)}} \, e^{2
                  \pi i \bracket{l}{\log\bigl(p(h \eta^{-1})\bigr)}}
                \, 1_\PFD \bigl( q(h \eta^{-1})
                \bigr). \label{TFShift} 
	\end{align}
Hence, we need to show that
	\begin{align}
		0  = \int_{\PID \rquo G} e^{2 \pi i
                  \bracket{\coAd(h^{-1})l}{\log(\vartheta)-
                    \log(\vartheta')} } &\, e^{2 \pi i
                  \bracket{l}{\log\bigl(p(h  \eta^{-1})\bigr) -
                    \log\bigl(p(h \, {\eta'}^{-1})\bigr)}} \nn \\ 
		\times \, 1_\PFD\bigl( q(&h \eta \inv ) \bigr) \, 1_\PFD \bigl(
                q(h {\eta'}^{-1}) \bigr) \,d\mu_{\PID \rquo G} \bigl(
                q(h) \bigr). \label{Ortho2} 
	\end{align}

\emph{Orthogonality when $\eta \neq \eta '$:} We focus on the
characteristic functions in \eqref{Ortho2} and see that 
$$1_F (q(h
\eta \inv )) 1_F (q(h
{\eta '} \inv )) \neq 0 \quad \Leftrightarrow \quad q(h) \in Fq(\eta) \cap
Fq(\eta ') = q(\Sigma \eta \cap \Sigma \eta ') \, .
$$
 Since $\eta , \eta '$
are in the quasi-lattice $\Gamma '$ and $\Sigma $ is a fundamental
domain for $\Gamma '$, we conclude that either  $1_F (q(h
\eta \inv )) 1_F (q(h{\eta '} \inv )) = 0$ or $\eta = \eta
'$. Consequently,  
$\langle \pi (\vartheta \eta \inv ) 1_F, \pi (\vartheta ' {\eta '} \inv )
1_F\rangle = 0$ for all $\vartheta , \vartheta ' \in \Gamma _M$ and $\eta
\neq \eta ' \in \Gamma _H$. 

\emph{Orthogonality when $\eta = \eta '$ and $\vartheta \neq \vartheta '$:}  

If $\eta = \eta '$, then 
	\begin{align*}
		e^{2 \pi i \bracket{l}{\log\bigl(p(h \eta)\bigr) -
                    \log\bigl(p(h \eta')\bigr)}} = 1 \, , 
	\end{align*}
and  the integral in \eqref{Ortho2} reduces to
	\begin{align*}
 \int_{\PID \rquo G} e^{2 \pi i
   \bracket{\coAd(h^{-1})l}{\log(\vartheta) - \log(\vartheta')} }
 \hspace{2pt} 1_\PFD \bigl( q(&h\eta \inv ) \bigr) \,d\mu_{\PID \rquo G}
 \bigl( q(h) \bigr) \, .
	\end{align*}
For further simplification, note that $\Ad(h)\log(\vartheta),
\Ad(h)\log(\vartheta') \in \pid$ for all $h \in \QA$ and $l$ vanishes
on $[\pid, \pid]$. Thus
	\begin{align}
		\bracket{l}{\Ad(h)\log(\vartheta) - \Ad(h)\log(\vartheta')} &= \bracket{l}{\Ad(h)\log(\vartheta) * \Ad(h)\log({\vartheta'}^{-1})} \nn \\
		&= \bracket{l}{\Ad(h)\log(\vartheta
                  {\vartheta'}^{-1})} \, , \label{Mult}
	\end{align}
and after  the change of variables $h \mapsto h \eta^{-1} =: h'$ we
need to  compute the integral
	\begin{align}
		\int_{\PFD} e^{2 \pi i \bracket{\coAd(\eta^{-1} {h'}^{-1})l}{\log(\vartheta {\vartheta'}^{-1})} } \,d\mu_{\PID \rquo G} \bigl( q(h') \bigr). \label{Ortho3}
	\end{align}

We now switch to  the coordinates given by the parametrization of $l
\cdot G$ in Theorem~\ref{AdaptedChevalleyRosenlicht} and use the Haar measure $\mu'_{\PID \rquo G}$ on $\PID \rquo G$ determined by the basis $\bigl\{ K^{2^d} \tilde{X}_1, \ldots, K^2 \tilde{X}_d \bigr\}$. In this basis we have
	\begin{align}
		\int_{\PID \rquo G} f\bigl(q(h) \bigr) \,d\mu'_{\PID \rquo G} \bigl(q(h) \bigr) = 
				\int_{\R^d} f \Bigl( q \bigl( \exp(t_1  K^{2^d} \tilde{X}_1) \cdots \exp(t_d  K^{2} \tilde{X}_d) \bigr) \Bigr) \,dt_1 \ldots \,dt_d \nn 
	\end{align}
and the integral over $\PFD$ becomes
	\begin{align}
		\int_{\PFD} \ldots \hspace{8pt} d\mu'_{\PID \rquo G}
                \bigl(q(h) \bigr) = \int_{[0, 1)^d} \ldots
                \hspace{8pt} dt_1 \ldots \,dt_d \, . \nn 
	\end{align}
Note that for the verification of the orthogonality of the $\pi(\gamma)1_\PFD$ the choice of the Haar measure on $\PID \rquo G$ is not relevant (because different Haar measures differ only by a multiplicative constant).

Let us first focus on $\coAd(\eta^{-1} {h'}^{-1})l$. Due to Lemma~\ref{PolyQL} (i), there exist an $m_{h' \eta} \in \PID$ and polynomials $\tilde{S}_{\tilde{X}_j}$ such that
	\begin{align}
		\eta^{-1}  h'^{-1} &= 
                (h' \eta)^{-1} = \Bigl( \exp(t_1  K^{2^d} \tilde{X}_1)
                \dots \exp(t_d  K^{2} \tilde{X}_d) \exp(\eta _d   K^{2} \tilde{X}_d)
                \cdots  \exp(\eta_1  K^{2^d} \tilde{X}_1) \Bigr)^{-1} \nn \\
		&= \biggl( m_{h' \eta} \exp\Bigl( \bigl(t_d + \eta_d + \tilde{S}_{\tilde{X}_d}(\eta_{d-1}, \ldots, t_1)\bigr) K^2 \tilde{X}_d \Bigr) \cdots \exp\Bigl((t_1 + \eta_1) K^{2^d} \tilde{X}_1 \Bigr) \biggr)^{-1} \nn \\ 
		&= \exp\Bigl(-(t_1 + \eta_1) K^{2^d} \tilde{X}_1
                \Bigr) \cdots \exp\Bigl( -\bigl(t_d + \eta_d +
                \tilde{S}_{\tilde{X}_d}(\eta_{d-1}, \ldots, \eta _1,
                t_{d-1}, \dots , t_1)\bigr)
                K^2 \tilde{X}_d \Bigr) m_{h' \eta}^{-1}. 
	\end{align}
Since $\PID$ is normal in $G$, there furthermore exists some $m'_{h' \eta} \in \PID$ such that
	\begin{align*}
		\eta^{-1}  h'^{-1} &= m'_{h' \eta} \exp\Bigl(-(t_1 +
                \eta_1) K^{2^d} \tilde{X}_1 \Bigr) \cdots \exp\Bigl(
                -\bigl(t_d + \eta_d +
                \tilde{S}_{\tilde{X}_d}(\eta_{d-1}, \ldots, t_1)\bigr)
                K^2 \tilde{X}_d \Bigr) \\
&:=  m'_{h' \eta} \exp(-s_1 K^{2^d} \tilde{X}_1) \cdots \exp(-s_d K^2 \tilde{X}_d).
	\end{align*}
At this point we  apply the parametrization of $l\cdot G$ of
Theorem~\ref{AdaptedChevalleyRosenlicht} and  obtain
	\begin{align}
		l \cdot \eta^{-1}  h'^{-1} &= l + \Bigl( K^2 \bigl( t_d + \eta_d + \tilde{S}_{\tilde{X}_d}(\eta_{d-1}, \ldots, t_1) + \tilde{Q}_d(s_{d-1}, \ldots, s_1) \bigr) \Bigr) Y^*_d + \ldots + \nn \\
		&\hspace{20pt} + \Bigl( K^{2^{d+1-j}} (t_j + \eta_j + \tilde{S}_{\tilde{X}_j}(\eta_{j-1}, \ldots, t_1) + \tilde{Q}_j(s_{j-1}, \ldots, s_1) \Bigr) Y^*_j + \ldots + \nn \\
		&\hspace{20pt} + K^{2^d} (t_1 - \eta_1) Y^*_1 . \label{ExpLHSCoord}
	\end{align}

For  the $\log$-terms  we employ Lemma~\ref{PolyQL} (ii) to rewrite
	\begin{align}
		\vartheta {\vartheta'}^{-1} &= \exp(\vartheta_1
                K^{-2^d} Y_1) \cdots \exp(\vartheta_d K^{-2} Y_d)
                \exp(-\vartheta'_d K^{-2} Y_d) \cdots
                \exp(-\vartheta_1 ' K^{-2^d} Y_1) \nn \\
		&= \exp\Bigl( \bigl( \vartheta_1 - \vartheta'_1 + \tilde{T}_{Y_1}(\vartheta_2, \ldots, \vartheta_d, -\vartheta'_2, \ldots, -\vartheta'_d) \bigr) K^{-2^d} Y_1 \Bigr) \cdots \nn \\
		&\hspace{20pt} \cdot \exp\Bigl( \bigl( \vartheta_j - \vartheta'_j + \tilde{T}_{Y_j}(\vartheta_{j+1}, \ldots, \vartheta_d, -\vartheta'_{j+1}, \ldots, -\vartheta'_d) \bigr) K^{-2^{d+1-j}} Y_j \Bigr) \cdots \nn \\
		&\hspace{20pt} \cdot \exp\Bigl( (\vartheta_d -
                \vartheta'_d) K^{-2} Y_d \Bigr) z_{\theta \theta '}\\
&:= 
z_{\theta \theta '} \, \exp(\zeta_1 K^{-2^d} Y_1) \cdots
\exp(\zeta_d K^{-2} Y_d) \label{ExpRHSCoord1} 
	\end{align}
for some $z_{\theta \theta '} \in Z(G)$. 

The conversion of strong Malcev coordinates to exponential coordinates
(cf.~Lemma~\ref{LemPropStrongMC}) yields 
	\begin{align}
		\log \bigl( \vartheta {\vartheta'}^{-1} \bigr) &= \Bigl( \vartheta_1 - \vartheta'_1 + \tilde{T}_{Y_1}(\vartheta_2, \ldots, -\vartheta'_d) + \tilde{R}_{Y_1}(\zeta_2, \ldots, \zeta_d) \Bigr) K^{-2^d} Y_1 + \ldots + \nn \\
		&\hspace{20pt} + \Bigl( \vartheta_j - \vartheta'_j + \tilde{T}_{Y_j}(\vartheta_{j+1}, \ldots, -\vartheta'_d) + \tilde{R}_{Y_j}(\zeta_{j+1}, \ldots, \zeta_d) \Bigr) K^{-2^{d+1-j}} Y_j + \ldots + \nn \\
		&\hspace{20pt} + (\vartheta_d  - \vartheta'_d) K^{-2^d} Y_d. \label{ExpRHSCoord2}
	\end{align}

Now assume that $\vartheta \neq \vartheta '$ and let $j$ be the largest
index such that $\vartheta _j \neq \vartheta _j '$ and thus $\vartheta _d =
\vartheta _d', \vartheta _{d-1} = \vartheta _{d-1} ', \vartheta _{j+1} = \vartheta
_{j+1}'$.   

Since $\tilde{T}_{Y_j}(\vartheta_{j+1}, \ldots, \vartheta_d,
-\vartheta_{j+1}, \ldots, -\vartheta_d) = 0  $ and thus $\zeta _{j+1}
= \dots \zeta _d = 0$, Lemma~\ref{LemPropStrongMC} implies that
	\begin{align} \label{tues1}
		\log \bigl( \vartheta {\vartheta'}^{-1} \bigr) &= \Bigl( \vartheta_1 - \vartheta'_1 + \tilde{T}_{Y_1}(\vartheta_2, \ldots, -\vartheta'_d) + \tilde{R}_{Y_1}(\zeta_2, \ldots, \zeta_d) \Bigr) K^{-2^d} Y_1 + \ldots + \nn \\
		&\hspace{20pt} + \Bigl( \vartheta_j - \vartheta'_j
                \Bigr) K^{-2^{d+1-j}} Y_j \, .
	\end{align}

Combining~\eqref{ExpLHSCoord} and~\eqref{tues1}, the action of $l\cdot \eta \inv
{h'}\inv $ on $\log (\vartheta {\vartheta '}\inv)$ yields
\begin{align*}
\lefteqn{ \langle l\cdot \eta \inv
{h'}\inv , \log (\vartheta {\vartheta '}\inv)
\rangle= } \\
&=   \Big(t_j + \eta _j + \tilde{S}_{\tilde{X}_j}(\eta_{j-1},
                  \ldots, t_1) + \tilde{Q}_d(s_{j-1}, \ldots, s_1)
                  \bigr) (\vartheta_j - \vartheta'_j) )  + \quad \text{ terms
 in} (t_1, \dots , t_{j-1} ) \, .
\end{align*}
So,
	\begin{align*}
		\int_{\PFD} e^{2 \pi i \bracket{\coAd(\eta^{-1}
                    {h'}^{-1})l}{\log(\vartheta {\vartheta'}^{-1})} }&
                \,d\mu'_{\PID \rquo G} \bigl(q(h') \bigr) = \\ 
                &
                \int_{[0, 1)^{j-1}} \int_{[0, 1)} e^{2 \pi i \bigl(
                  t_j + \eta_j + \tilde{S}_{\tilde{X}_j}(\eta_{j-1},
                  \ldots, t_1) + \tilde{Q}_d(s_{j-1}, \ldots, s_1)
                  \bigr) (\vartheta_j - \vartheta'_j)} \,dt_j \\ 
		&\hspace{10pt}  \times \mbox{terms in} \, (t_{j-1}, \ldots, t_1) \,dt_{j-1} \ldots \,dt_1 = \\
                \int_{[0, 1)^{j-1}}& \int_{[0, 1)} e^{2 \pi i t_j
                  (\vartheta_j - \vartheta'_j)} \,dt_j \times e^{2 \pi
                  i \bigl( \eta_j +
                  \tilde{S}_{\tilde{X}_j}(\eta_{j-1}, \ldots, t_1) +
                  \tilde{Q}_d(s_{j-1}, \ldots, s_1)  \bigr)
                  (\vartheta_j - \vartheta'_j)} \\ 
		&\hspace{10pt}  \times \mbox{terms in} \, (t_{j-1},
                \ldots, t_1) \,dt_{j-1} \ldots \,dt_1 \, .
	\end{align*}
Since $\int _0^1 e^{2\pi i  t_j  (\vartheta _j - \vartheta _j ')} dt_j =
0$ when $\vartheta _j \neq \vartheta _j '$, we have proved that the
functions $\pi (\vartheta \eta  \inv ) 1_F $ and  $\pi (\vartheta ' \eta  \inv ) 1_F $ are orthogonal.

\emph{Normalization:} If $\gamma = \gamma '$, then $\subbracket{\L{2}{M \rquo G}}{\pi(\gamma ) 1_F}{\pi (\gamma ' ) 1_F} = \|1_F\|^2_{\L{2}{M \rquo G}} = \mu_{M \rquo G}(F)$. The basis change $S_l$ from $\{ Z_1, \ldots, Y_d,
X_1, \ldots, X_d \}$ to $\{ Z_1, \ldots, Y_d, \tilde{X}_d, \ldots,
\tilde{X}_1 \}$ from the proof of Proposition~\ref{Subgroup} acts
trivially on $\pid$. Since the according change of Haar measure obeys
$\widetilde{\mu}_{G/Z(G)} = \det(S_l)^{-1} \hspace{2pt} \mu_{G/Z(G)} =
\fd \hspace{2pt} \mu_{G/Z(G)}$, we thus have $$\mu _{\PID \rquo G}
\circ q = \fd^{-1} \hspace{2pt} \bigl( \tilde{\mu}_{\PID \rquo G}
\circ q \bigr).$$ Moreover, since $\tilde{\mu}_G(\Sigma_G) = 1$ for
$\Sigma_G = \{ \exp(\tilde{x}_{1} \tilde{X}_1) \cdots \exp(z_1 Z_1)
\mid z_1, \ldots, \tilde{x}_{1} \in [0,1) \}$ 
and
$$F =\Bigl\{ q \Bigl( \exp(t_1  K^{2^d} \tilde{X}_1) \dots \exp(t_d  K^{2} \tilde{X}_d) \Bigr) \mid t_1, \ldots, t_d \in [0, 1)^d \Bigr\}$$ by Proposition~\ref{Subgroup}~$(i)$ and $(ii)$, respectively, we have $\tilde \mu _{M \rquo G}(F) = \prod _{j=0}^{d-1}
K^{2^{d-j}} = K^{2(2^d-1)}=  C^2$ and, consequently,
$$
\mu _{M \rquo G} (F) = d_\pi \inv \tilde \mu _{M \rquo G}(F) = d_\pi \inv C^2.
$$
This yields the normalization of the system.

	\end{proof}

	\begin{prop} \label{Completeness} [Step 3. Completeness.]

\noindent Let $\Gamma$ be the discrete subset of  $G/Z(G)$ from
Definition~\ref{DiscreteSet}. 
Then the
orthogonal  system 
	\begin{align*}
		\{ 
                \pi(\gamma) \, 1_\PFD \mid \gamma \in \Gamma \}
	\end{align*}
of $\L{2}{\PID \rquo G}$ from Proposition~\ref{Orthonormality} is complete.
	\end{prop}

	\begin{proof}
By Corollary~\ref{QLQuoGr}  $q(\Gamma')$ is a quasi-lattice (a uniform subgroup)  of $\PID
\rquo G$ with fundamental domain $F = q(\Sigma )$,  
the sets $F q(\eta ) =
q(\Sigma \eta ), \eta \in \Gamma '_{\QA}$, form a partition of $M \rquo G$. It suffices to show that for
each $\eta \in \Gamma'_\QA$ the orthogonal  subsystem $\{ 
\pi(\vartheta \eta^{-1}) \,
1_\PFD \mid \vartheta \in \Gamma'_\PID \}$ is complete in $\L{2}{q(
  \Sigma \eta )}$. Thus, fix an arbitrary $\eta \in \Gamma '_{\QA}$. Combining \eqref{TFShift} and the change of variables $h \mapsto h \eta^{-1} =: h'$ from the proof of Proposition~\ref{Orthonormality},
we have 
	\begin{align*}
		\Bigl( \pi(\gamma)1_\PFD \Bigr) \big(q(h) \bigr) = e^{2 \pi i \bracket{\coAd (\eta^{-1}
                    h'^{-1})l}{\log(\vartheta)}} \, e^{2 \pi i
                  \bracket{l}{\log\bigl(p(h')\bigr)}} \, 1_\PFD \bigl(
                q(h') \bigr). 
	\end{align*}
Since the map $f(h) \mapsto e^{-2 \pi i
                  \bracket{l}{\log\bigl(p(h')\bigr)}} f(h) = e^{-2 \pi i
                  \bracket{l}{\log\bigl(p(h\eta \inv )\bigr)}} f(h) $  is
                unitary on $L^2(Fq(\eta ))$, it suffices to
                show that the set
	\begin{align*}
		\{ e^{2 \pi i \bracket{\coAd (\eta^{-1}
                    h'^{-1})l}{\log(\vartheta)}} \,  1_\PFD \bigl(
                q(h') \bigr) \mid \vartheta \in \Gamma _\PID ' \}
	\end{align*}
is complete in $\L{2}{\PFD}$. For  $f \in \L{2}{\PFD}$  this means that
	\begin{align}
		\int_{\PFD} f \bigl(q(h') \bigr) \hspace{2pt} e^{2 \pi i \bracket{\coAd (\eta^{-1} h'^{-1})l}{\log(\vartheta)}} \,d\mu'_{\PID \rquo G} \bigl(q(h') \bigr) = 0 \hspace{10pt} \forall \vartheta \in \Gamma _\PID '
		\label{CheckCompl}
	\end{align}
implies $f = 0$ a.e. Note that by \eqref{ExpLHSCoord}, we have
	\begin{align*}
		l \cdot \eta^{-1}  h'^{-1} &= l + \sum_{j=1}^d \Bigl( t_j + \eta_j + P_j(\eta_{j-1}, \ldots, \eta_1, t_{j-1}, \ldots, t_1) \Bigr) K^{2^{d+1-j}} Y^*_j 
	\end{align*}
for some polynomials $P_j$, and by \eqref{ExpRHSCoord2} with $\theta '
= 0$, we have
	\begin{align*}
		\log \bigl( \vartheta \bigr) = \sum_{j=1}^d \Bigl( \vartheta_j + \tilde{R}_{Y_j}(\vartheta_{j+1}, \ldots, \vartheta_d) \Bigr) K^{-2^{d+1-j}} Y_j.
	\end{align*}
Consequently,
	\begin{align}
		e^{2 \pi i \bracket{\coAd (\eta^{-1} h'^{-1})l}{\log(\vartheta)}} = \prod_{j=1}^d e^{2 \pi i \bigl( t_j + \eta_j + P_j(\eta_{j-1}, \ldots, t_1) \bigr) \bigl( \vartheta_j + \tilde{R}_{Y_j}(\vartheta_{j+1}, \ldots, \vartheta_d) \bigr)}. \label{BVCoord}
	\end{align}
To verify \eqref{CheckCompl}, we therefore switch to the coordinates
	\begin{align*}
		\psi_\QA(t_1, \ldots, t_d) := q \bigl( \exp(t_1  K^{2^d} \tilde{X}_1) \dots \exp(t_d  K^{2} \tilde{X}_d) \bigr)
	\end{align*}
from $[0,1)^d$ to $\PFD$. Substituting \eqref{BVCoord} into
\eqref{CheckCompl}, it is equivalent to show that
	\begin{align}
		\int_{[0, 1)^d} f(t_1, \ldots, t_d) \hspace{2pt}  \prod_{j=1}^d e^{2 \pi i \bigl( t_j + \eta_j + P_j(\eta_{j-1}, \ldots, t_1) \bigr) \bigl( \vartheta_j + \tilde{R}_{Y_j}(\vartheta_{j+1}, \ldots, \vartheta_d) \bigr)} \,dt_1 \ldots \,dt_d = 0
		\hspace{10pt} \forall \vartheta \in \Gamma _\PID '
		\label{CheckCompl2}
	\end{align}
implies $f = 0$ a.e.

Thus, suppose \eqref{CheckCompl2} holds true. First, we may omit the constant factors $e^{2 \pi i \eta_j \bigl( \vartheta_j + \tilde{R}_{Y_j}(\vartheta_{j+1}, \ldots, \zeta_d) \bigr)}$ of modulus one. Since $P_1 = 0$, we may rewrite \eqref{CheckCompl2} as
	\begin{align*}
		\int_0^1 \Bigl[ \int_{[0, 1)^{d-1}} f(t_1, \ldots,
                t_d) \hspace{2pt} & \prod_{j=2}^d e^{2 \pi i \bigl( t_j
                  + P_j(\eta_{j-1}, \ldots, \eta_1, t_{j-1}, \ldots,
                  t_1) \bigr) \bigl( \vartheta_j +
                  \tilde{R}_{Y_j}(\vartheta_{j+1}, \ldots,
                  \vartheta_d) \bigr)} \,dt_2 \ldots \,dt_d \Bigr] \\ 
		& \hspace{10pt} \times \, e^{2 \pi i t_1
                  \tilde{R}_{Y_1}(\vartheta_2, \ldots, \vartheta_d) }
                \hspace{2pt} e^{2 \pi i t_1 \vartheta_1} \,dt_1 = 0. 
	\end{align*}
As this holds true for all $\vartheta_1 \in \Z$, it follows that
	\begin{align*}
		0 &= \int_{[0, 1)^{d-1}} f(t_1, \ldots, t_d) \hspace{2pt} \prod_{j=2}^d e^{2 \pi i \bigl( t_j + P_j(\eta_{j-1}, \ldots, \eta_1, t_{j-1}, \ldots, t_1) \bigr) \bigl( \vartheta_j + \tilde{R}_{Y_j}(\vartheta_{j+1}, \ldots, \vartheta_d) \bigr)} \,dt_2 \ldots \,dt_d \\
		&\hspace{15pt} \times e^{2 \pi i t_1 \tilde{R}_{Y_1}(\vartheta_2, \ldots, \vartheta_d)} \\
		&= e^{2 \pi i t_1 \tilde{R}_{Y_1}(\vartheta_2, \ldots, \vartheta_d)} \hspace{2pt} e^{2 \pi i P_2(\eta_1, t_1) \bigl( \vartheta_2 + \tilde{R}_{Y_2}(\vartheta_3, \ldots, \vartheta_d) \bigr)} \hspace{2pt} \int_{[0, 1)^{d-1}} f(t_1, \ldots, t_d) \hspace{2pt}  \\
		&\hspace{10pt} \times  \prod_{j=3}^d e^{2 \pi i \bigl(
                  t_j + P_j(\eta_{j-1}, \ldots, \eta_1, t_{j-1},
                  \ldots, t_1) \bigr) \bigl( \vartheta_j +
                  \tilde{R}_{Y_j}(\vartheta_{j+1}, \ldots,
                  \vartheta_d) \bigr)} \hspace{2pt} e^{2 \pi i t_2
                  \vartheta_2} \,dt_3 \ldots \,dt_d \,dt_2 
	\end{align*}
for almost all $t_1 \in [0,1)$. 
The unimodular factor outside the integral can be deleted and we apply the same argument to the remaining integrals. Repeating this procedure (induction), we arrive at
	\begin{align}
		0 &=\int_{[0, 1)^{d-k}} f(t_1, \ldots, t_d) \hspace{2pt} \prod_{j=k+2}^d e^{2 \pi i \bigl( t_j + P_j(\eta_{j-1}, \ldots, \eta_1, t_{j-1}, \ldots, t_1) \bigr) \bigl( \vartheta_j + \tilde{R}_{Y_j}(\vartheta_{j+1}, \ldots, \vartheta_d) \bigr)} \nn \\
		&\hspace{15pt} \times e^{2 \pi i t_{k+1} \vartheta_{k+1}} \,dt_{k+2} \ldots \,dt_d \,dt_{k+1} \label{CheckCompl3}
	\end{align}
for all $\vartheta_{k+1} \in \Z$ and almost all $(t_1, \dots ,
t_{k})\in [0,1)^k$. Consequently,
	\begin{align*}
		0 &= \int_{[0, 1)^{d-k-1}} f(t_1, \ldots, t_d) \hspace{2pt} \prod_{j=k+2}^d e^{2 \pi i \bigl( t_j + P_j(\eta_{j-1}, \ldots, \eta_1, t_{j-1}, \ldots, t_1) \bigr) \bigl( \vartheta_j + \tilde{R}_{Y_j}(\vartheta_{j+1}, \ldots, \vartheta_d) \bigr)} \,dt_{k+2} \ldots \,dt_d \\
		&= e^{2 \pi i t_{k+2} \tilde{R}_{Y_{k+2}}(\vartheta_{k+3}, \ldots, \vartheta_d)} \hspace{2pt} e^{2 \pi i P_{k+2}(\eta_{k+1}, \ldots, \eta_1, t_{k+1}, \ldots, t_1) \bigl( \vartheta_{k+2} + \tilde{R}_{Y_{k+2}}(\vartheta_{k+3}, \ldots, \vartheta_d) \bigr)} \\
		&\hspace{10pt} \times \int_{[0, 1)^{d-k-1}} f(t_1, \ldots, t_d) \hspace{2pt} \prod_{j=k+3}^d e^{2 \pi i \bigl( t_j + P_j(\eta_{j-1}, \ldots, \eta_1, t_{j-1}, \ldots, t_1) \bigr) \bigl( \vartheta_j + \tilde{R}_{Y_j}(\vartheta_{j+1}, \ldots, \vartheta_d) \bigr)} \\
		 &\hspace{10pt} \times e^{2 \pi i t_{k+2} \vartheta_{k+2}} \,dt_{k+3} \ldots \,dt_d \,dt_{k+2},
	\end{align*}
for almost all $(t_1, \dots ,
t_{k+1})\in [0,1)^{k+1}$,  
which in turn implies the next step
	\begin{align*}
		0 &=\int_{[0, 1)^{d-k-1}} f(t_1, \ldots, t_d) \hspace{2pt} \prod_{j=k+3}^d e^{2 \pi i \bigl( t_j + P_j(\eta_{j-1}, \ldots, \eta_1, t_{j-1}, \ldots, t_1) \bigr) \bigl( \vartheta_j + \tilde{R}_{Y_j}(\vartheta_{j+1}, \ldots, \vartheta_d) \bigr)} \\
		&\hspace{15pt} \times e^{2 \pi i t_{k+2} \vartheta_{k+1}} \,dt_{k+3} \ldots \,dt_d \,dt_{k+2}
	\end{align*}
for all $\vartheta_{k+2} \in \Z$. For the last step $k = d-1$ we obtain from \eqref{CheckCompl3}
	\begin{align*}
		\int_{[0, 1)^{d-(d-1)}} f(t_1, \ldots, t_d) \hspace{2pt} \prod_{j=(d-1)+2}^d \ldots \hspace{15pt} e^{2 \pi i t_d \vartheta_d} \,dt_d = 0
	\end{align*}
for all $\vartheta_d \in \Z$ and almost all $(t_1, \dots ,
t_{d-1})\in [0,1)^{d-1}$, from which we obtain $f(t_1, \ldots, t_d) =
0$ a.e. This proves the completeness.

	\end{proof}

In the case when  $\Gamma$ is a quasi-lattice or  a uniform subgroup,
we can give  a
shorter and more elegant proof of completeness. In particular, the
following observation yields a more structural proof of
Theorem~\ref{TVMainThm} under the assumptions (i) --- (v). 

	\begin{lem}
\noindent Let $G$ be a nilpotent Lie group. Let $\Gamma$ be a
quasi-lattice in $G/Z(G)$ with fundamental domain $\Sigma$, $\pi \in 
SI/Z(G)$ and  $w \in \RS$ with $\Norm{\RS}{w} = 1$,  such that
  the set $\{\pi(\gamma)w \mid \gamma \in \Gamma \}$ forms an orthonormal system in
$\RS$. Then the following are equivalent:
	\begin{itemize}
		\item[(i)] $\{ \pi(\gamma)w \mid \gamma \in \Gamma \}$
                  is complete in $\HS _\pi$, and thus $\{\pi(\gamma)w \mid \gamma \in \Gamma \}$ is an \onb .
		\item[(ii)] $\Gmeas{G/Z(G)}{\Sigma}^{-1} = \fd$.
	\end{itemize}

	\end{lem}

	\begin{proof}
$(i) \Rightarrow (ii)$ If $\{ \pi(\gamma)w \mid \gamma \in \Gamma \}$
is an orthonormal basis, then for every $v \in \RS$ and $ \dot{g} \in G/Z(G)$ we have 
	\begin{align*}
		\Norm{\RS}{v}^2 = \Norm{\RS}{\pi(\dot{g})^*v}^2 =
		\sum_{\gamma \in \Gamma}
		\Abs{\subbracket{\RS}{\pi(\dot{g})^*v}{\pi(\gamma)w}}^2= \sum_{\gamma \in \Gamma}
		\Abs{\subbracket{\RS}{v}{\pi(\dot{g} \gamma)w}}^2 . 
	\end{align*}
Hence, square-integrability yields
	\begin{align*}
		\fd^{-1} \, \Norm{\RS}{v}^2 &= \int_{G/Z(G)} \Abs{\subbracket{\RS}{v}{\pi(\dot{g})w}}^2 \,d\dot{g} \\
		&= \int_{\Sigma} \sum_{\gamma \in \Gamma} \Abs{\subbracket{\RS}{v}{\pi(\dot{g} \gamma)w}}^2 \,d\dot{g} \\
		&= \int_{\Sigma} \Norm{\RS}{v}^2 \,d\dot{g} \\
		&= \Gmeas{G/Z(G)}{\Sigma} \, \Norm{\RS}{v}^2.
	\end{align*}

$(ii) \Rightarrow (i)$ By Bessel's inequality we have
	\begin{align*}
		\Norm{\RS}{v}^2 = \Norm{\RS}{\pi(\dot{g})^*v}^2 \geq \sum_{\gamma \in \Gamma} \Abs{\subbracket{\RS}{v}{\pi(\dot{g} \gamma)w}}^2
	\end{align*}
for all $\dot{g} \in \Sigma$. Since $\dot{g} \mapsto \sum_{\gamma \in \Gamma}
\Abs{\subbracket{\RS}{v}{\pi(\dot{g} \gamma)w}}^2$ is in
$\L{2}{\Sigma, d\dot{g}}$ and $d_\pi = \Gmeas{G/Z(G)}{\Sigma}^{-1}$, we can therefore estimate 
	\begin{align*}
		0 &= \Norm{\RS}{v}^2 - \fd \int_{G/Z(G)} \Abs{\subbracket{\RS}{v}{\pi(\dot{g})w}}^2 \,d\dot{g} \\
		&= \fd \int_\Sigma \Norm{\RS}{v}^2 \,d\dot{g} - \fd \,  \int_{\Sigma} \sum_{\gamma \in \Gamma} \Abs{\subbracket{\RS}{v}{\pi(\dot{g} \gamma)w}}^2 \,d\dot{g} \\
		&= \fd \, \int_{\Sigma} \left( \Norm{\RS}{v}^2 -
		\sum_{\gamma \in \Gamma}
		\Abs{\subbracket{\RS}{v}{\pi(\dot{g} \gamma)w}}^2 \right) \,d\dot{g}.
	\end{align*}
Since the integrand is non-negative, this  implies that
$\Norm{\RS}{v}^2 - \sum_{\gamma \in \Gamma}
\Abs{\subbracket{\RS}{v}{\pi(\dot{g} \gamma)w}}^2 = 0$ for almost all
$\dot{g} \in \Sigma$, thus $\{ \pi(\gamma)w \mid \gamma \in \Gamma \}$ is
complete. 
	\end{proof}

\subsection{Graded $SI/Z$-Groups} \label{GrSIZGrSubs}

In this subsection we show that the technical assumptions of
Theorem~\ref{TVMainThm} are satisfied for graded $SI/Z$-groups with
one-dimensional center.   
We first prove several properties of graded groups. 

In the following $G$ is a graded $SI/Z$-group
of $\dim(G) = r + 2d$ with $r$-dimensional center $Z(G)$. Let $\Lie{g}
= \bigoplus_{k=1}^N \Lie{g}_k$ be a  gradation of $\Lie{g}$  with the
understanding that $\Lie{g}_N \neq \{0\}$ and  $N$ is the smallest index
such that $\Lie{g}_{N + k} = \{ 0 \}$ for all $k \in \N$.

	\begin{lem} \label{thurs3}
Let $\Lie{g}$ be a graded Lie algebra with gradation $\Lie{g} =
\bigoplus_{k=1}^N \Lie{g}_k$. 

(i) Then $\Lie{n}_{k_0} = \bigoplus _{k=k_0}^N \Lie{g}_k$ is an ideal of
$\Lie{g}$ for each  $k_0 = 1, \dots , N$. 

(ii)  $\Lie{g}_N \subseteq \Liez{g}$.  In
particular, if $\dim(\Liez{g}) = 1$, then $\Lie{g}_N = \Liez{g}$. 

(iii)  Every basis of $\Lie{g}$ that is  a union of bases of the
$\Lie{g}_k$ is a strong Malcev basis for $\Lie{g}$. 
	\end{lem}

	\begin{proof}
(i) follows directly from the definition of a gradation. 

(ii) For every $X \in \Lie{g}_k$, $1 \leq k \leq N$, and every $Y \in \Lie{g}_N$ we have $[X, Y] \in \Lie{g}_{N+k} = \{ 0 \}$, hence $Y \in \Liez{g}$.
If $\dim(\Liez{g}) = 1$, then $1 \leq \dim(\Lie{g}_N) \leq
\dim(\Liez{g}) = 1$, so the subspaces must coincide.

(iii) 	Let $\{X_{k_1}, \ldots, X_{k_d}\}$ be a basis of  $\Lie{g}_k \neq \{ 0
\}$ and $X \in \Lie{g}$ be arbitrary. Then    $[X, X_{k_j}]
\in \bigoplus _{k'=1} ^N \Lie{g}_{k + k'}$. Consequently $
\Rspan{X_{k_j}, \ldots X_{k_d}} \bigoplus_{l = k+1}^N  \Lie{g}_k$ is
an ideal,  which proves the strong Malcev
property of the union of bases. 
	\end{proof}

Under the assumption  $\dim(Z(G)) = r = 1$, we construct  polarizing
ideals $\pid \nsubgr \Lie{g}$ which are subordinate to all $l \in
\Liez{g}^*$ simultaneously. This implies that 
all $\pi \in SI/Z(G)$  can be induced from the same polarization $\pid$. 
We write $\Lie{g} = \pid \oplus \qa$  and we  denote by $\pr_{\Lie{g}_k}$
the natural projection $\Lie{g} \to \Lie{g}_k$. We distinguish  the
cases of even and odd $N$. 

	\begin{prop} [N odd] \label{Nodd}
Let $G$ be a graded $SI/Z$-group with one-dimensional center and  gradation $\Lie{g} =
\bigoplus_{k=1}^N \Lie{g}_k$ with $N = 2 N_0 +1$, $N_0 \in \N$.
Then
$\pid := \bigoplus_{k=N_0+1}^N \Lie{g}_{k}$ is an  ideal of $\Lie{g}$
that  is polarizing for all $l \in \Liez{g}^*$. The normal subgroup
$\PID= \exp (\pid )$ therefore  induces all $\pi
\in SI/Z(G)$. 
	\end{prop}

	\begin{proof}
Let $l = \lambda Z^*  \in \Liez{g}^* $, $\lambda \neq
0$. By Lemma~\ref{thurs3}(i) the subspace $\bigoplus_{k= k_0}^N
\Lie{g}_k \subseteq \Lie{g}$ is  an ideal of $\Lie{g}$ for
every $1 \leq k_0 \leq N$. Since $[\Lie{g}_k, \Lie{g}_{k'}] \subseteq
\Lie{g}_{k + k'} = \{ 0 \}$ for all $k, k' \geq N_0+1$, 
$\pid$ must be  an Abelian ideal. Consequently, $\pid$ is subordinate to
$l$. 

We show that $ \dim(\Lie{g}_{N-k})= \dim(\Lie{g}_{k})$. Choose a basis $\{V_1, \ldots, V_p
\} $ for $\Lie{g}_{N-k}$ and $\{W_1, \ldots, W_q\}$
 for $\Lie{g}_{k}$ and assume that $p\neq q$, $p<q$ say.  Then the
$p\times q$-matrix $C$ with entries 
$C_{jk} = l([V_j,W_k])$  possesses  a non-trivial kernel $c= (c_1,
\dots , c_q) \in \R ^q$. Set $Y = \sum _{m=1}^q c_m \hspace{2pt} W_m$, then 
$$l([V_j,Y])= \sum _{m=1}^q c_m \hspace{2pt} l([V_j,W_m]) = (Cc)_j = 0 \, .
$$
Consequently, $l([V,Y])= 0 $ for all $V \in \Lie{g}_{N-k}$ and, by the
properties of the gradation, $B_l(X,Y)= l([X,Y])= 0 $ for all $X \in
\Lie{g}$.   
This contradicts the assumption that $B_l$ restricted to $\Lie{g}/\Lie{z}$ is non-degenerate. 

Finally, $\dim(\ggl) = 1 + \sum _{k=1}^N \dim (\ggl _k ) = 1 + 2 \sum
_{k=1}^{N_0} \dim (\ggl _k ) $. It follows that $\dim(\pid) = \sum
_{k=N_0+1}^{N} \dim (\ggl _k) = d +1 $, and thus $\pid $ is a polarization.    


	\end{proof}

	\begin{prop} [N even] \label{Neven}
Let $G$ be a graded $SI/Z$-group with one-dimensional center and  with gradation $\Lie{g} =
\bigoplus_{k=1}^N \Lie{g}_k$ with $N = 2 N_0$, $N_0 \in \N$. Then
there exists an ideal $\pid \subseteq \Lie{n}_{N_0} = \bigoplus_{k=N_0}^N \Lie{g}_{k}
\subseteq \Lie{g}$ that is polarizing for all $l \in
\Liez{g}^*$. Again, the normal subgroup
$\PID= \exp (\pid )$   induces all $\pi
\in SI/Z(G)$.  
	\end{prop}

	\begin{proof}
Let $l= \lambda Z^* \in \Liez{g}^*, \lambda \neq 0 $. Again we obtain 
$\dim(\Lie{g}_{N-k})= \dim(\Lie{g}_{k}) $ for $k=1, \dots , N_0-1$. Since
$\dim(\Lie{z}(\ggl ))= \dim(\Lie{g}_{N})=1$ and $\dim( \ggl ) $ is odd,
the dimension $\dim(\ggl _{N_0}) $ must be even. If $\ggl _{N_0} =
\{0\}$, then nothing needs to be shown.  If $\dim (\ggl _{N_0}) =
2d_{N_0} \geq 2$, then $B_l\Big| _{\ggl _{N_0}\times \ggl _{N_0}} $,
the restriction of $B_l$ to $\ggl _{N_0}$,   is a  symplectic
form. Therefore $\ggl _{N_0}$ possesses a basis $\{ Y_1, \ldots,
Y_{d_{N_0}}, X_1, \ldots, X_{d_{N_0}} \}$, such that $B_l([Y_j,X_k]) =
\delta _{jk}$ and $B_l([Y_j,Y_k]) =B_l([X_j,X_k]) = 0$ for $j,k=1,
\dots d_{N_0}$. 

Now set  $\pid := \Rspan{Y_1, \ldots,
  Y_{d_{N_0}}} \bigoplus_{k=N_0 + 1}^N \Lie{g}_k$. Then $\pid $  is a
$d+1$-dimensional Abelian ideal of $\Lie{g}$ and  subordinate to $l$,
as claimed. 
	\end{proof}


We next show that the technical assumptions of Theorem~\ref{TVMainThm}
are satisfied for graded groups with one-dimensional center. The main
point is the existence of a Chevalley-Rosenlicht-admissible Malcev
basis.



	\begin{prop} \label{PropQL} [Quasi-Lattice]
Let $G$ be a graded $SI/Z$-group with one-dimensional center. Let $l
\in \Liez{g}^* \setminus \{ 0 \}$ and 
$\pid$ be the polarization constructed in  Proposition~\ref{Nodd} or \ref{Neven}. Then
$\Lie{g}$ 
possesses a Ch-R-admissible basis $\{ Z, Y_1, \ldots, Y_d,
\tilde{X}_d, \ldots, \tilde{X}_1 \}$ of $\Lie{g}$  which passes through
$\pid$ and the gradation and is subordinate to $l$. 

If, in addition, $\ggl $ possesses a rational structure compatible
with the gradation and $l=\lambda Z^*$ for $\lambda \in \Q , \lambda
\neq 0$, then this Ch-R-admissible basis generates the same rational
structure $\ggl _\Q =  \Qspan{ Z, Y_1, \ldots, Y_d,
\tilde{X}_d, \ldots, \tilde{X}_1 }$. 
	\end{prop}

	\begin{proof}
For $k = 1, \ldots, N_0$,  choose a basis $\{ V_1, \ldots, V_p \}$ for
$\Lie{g}_{N-k}$ and a basis $\{ W_1, \ldots, W_q \}$ for
$\Lie{g}_k$. We have already seen that $p = q$.  
The same argument as in Proposition~\ref{Nodd} shows that the $p
\times p$-matrix $C$ with entries $C_{j, k} =  l([W_j,
V_k])$ is invertible. Set $\tilde{W}_j = \sum_{k=1}^p (C \inv)_{j, k}
\hspace{2pt} W_k$, then 
	\begin{align} \label{br3}
		l([\tilde{W}_j, V_k]) = \delta_{j, k}.
	\end{align}
If, in addition, $N = 2 N_0$ is even, we partition the symplectic
basis of $\Lie{g}_{N_0}$ as in the proof of  Proposition~\ref{Neven}. 

Now take the union of bases $\{V_j \}$ of $\Lie{g}_{N-k}$, $k = 0, \ldots, N_0$, and relabel it as $\{Z, \ldots, Y_d \}$, and take the union of bases $\{ \tilde{W}_j \}_j$ of $\Lie{g}_k$, $k = 1, \ldots, N_0$, and relabel it as $\{ \tilde{X}_d, \ldots, \tilde{X}_1 \}$. By construction, we have
	\begin{align*}
		l([\tilde{X}_j, Y_k]) = \delta_{j, k}, \hspace{10pt} j, k = 1, \ldots, d.
	\end{align*}
According to Theorem~\ref{AdaptedChevalleyRosenlicht}, $\{ \tilde{X}_d, \ldots, \tilde{X}_1 \}$ is a basis that parametrizes the orbit $l \cdot G$. At the same time, $\{Z, \ldots, Y_d, \tilde{X}_d, \ldots, \tilde{X}_1 \}$ is a union of bases of the subspaces $\Lie{g}_k$. By Lemma~\ref{thurs3} (iii), this is a strong Malcev basis of $\Lie{g}$. This means that $\Lie{g}$ possesses a Ch-R-admissible basis subordinate to $l = \l Z^*$.

In the presence of a rational structure compatible with the gradation,
the matrix \linebreak  $\langle  Z^*, [W_j, V_k]\rangle_{j,k=1, \dots , p}$ has rational
entries. Thus for $l=\lambda Z^*, \lambda \in \Q$, the new basis
elements $\tilde W_j$ defined by \eqref{br3} are also in $\ggl _\Q$ and
thus $\ggl _\Q =  \Qspan{ Z, Y_1, \ldots, Y_d,
\tilde{X}_d, \ldots, \tilde{X}_1 }$.
	\end{proof}



	\begin{proof} [Proof of Theorem~\ref{MainThmQL}]
By Propositions~\ref{Nodd} and~\ref{Neven} every $l=\lambda Z^*, \lambda \neq
0$ possesses the same polarization $\Lie{m}$, which by
construction is an ideal of $\ggl $.  This is condition  (i) of 
Theorem~\ref{TVMainThm}. By  Proposition~\ref{PropQL} $\ggl $ possesses a
Ch-R-admissible basis subordinate to $l$ passing through $\Lie{m}$ and
$\Lie{z}$, which is condition (ii).  In conclusion,
Theorem~\ref{TVMainThm}  now yields the existence of an orthonormal
basis in the orbit of $\pi _l$. 
	\end{proof}

	\begin{proof} [Proof of Theorem~\ref{MainThmUnif}]
If $l=\lambda Z^*, \lambda \in \Q , \lambda \neq 0$, then by
Proposition~\ref{PropQL} $\ggl _\Q$ has a Ch-R-admissible basis that
is compatible with the gradation. Therefore the additional conditions
(iii) -- (v) of Theorem~\ref{TVMainThm} are satisfied. This implies
the existence of an orthonormal basis  with respect to a uniform
subgroup. In particular, for $l = \pm Z^*$ with the associated
representations $\pi _{\pm}$, there
exist a uniform subgroup $\Gamma _G \subseteq G$ and a set $F
\subseteq  M \rquo G$, such that $\{ \pi _{\pm} (\gamma )1_F \mid
\gamma \in \mathrm{pr} (\Gamma _G)\}$ is an orthogonal basis of
$L^2(M\ \rquo G)$. 

For  $l = \lambda Z^*$ with irrational $\lambda \neq 0$ we use
dilations. The dilation $\delta _s : \ggl \to \ggl $ in ~\eqref{dilgrad}
lifts to an automorphism $\delta _s : G \to G$ and leads to the
one-parameter family of irreducible unitary representations $\pi
_s = \pi _{\pm} \circ \delta _s$. These are irreducible and in
$SI/Z(G)$. Now let $\Gamma _s := \mathrm{pr}(\delta _{s\inv } (\Gamma
_G))$. Then $\Gamma _s$ is a uniform subgroup of $G/Z(G)$ with fundamental domain $\pr \bigl( \delta_{s \inv}(\Sigma) \bigr)$. If $\tilde
\gamma = \delta _{s\inv }(\gamma )$ for $\gamma \in \Gamma _G$, then
$\pi _s(\tilde \gamma ) = \pi _{\pm } \circ \delta _s (\delta _{s\inv
} \gamma ) = \pi _{\pm } (\gamma )$, and for $\tilde{h} \in \tilde{\PFD} := q\bigl( \delta_{s \inv} (\Sigma) \bigr)$ a short computation yields
	\begin{align*}
		\Bigl( \pi_s(\tilde{\gamma}) 1_{\tilde{\PFD}} \Bigr)\bigl( q( \tilde{h}) \bigr) = \Bigl( \pi_\pm(\gamma) 1_\PFD \Bigr)\bigl( q(h) \bigr).
	\end{align*}
Consequently, the sets 
$\{ \pi _{\pm} (\gamma )1_\PFD \mid
\gamma \in \mathrm{pr} (\Gamma _G)\}$ and $\{ \pi _{s} (\tilde \gamma )1_{\tilde{\PFD}} \mid
\tilde \gamma \in \Gamma _s\}$ coincide, and we have found an
orthonormal basis in the orbit of $\pi _s$. 

Finally, we note that the $\pi _s$ exhaust all representations in
$SI/Z(G)$: since
$$
\pi _{\pm } (\delta _s (e^{tZ}))  = \pi _{\pm } (e^{s^N tZ}) = e^{\pm
  2\pi i \langle Z^*, s^NtZ\rangle } =   e^{\pm
  2\pi i \langle s^N Z^*, tZ\rangle } \, ,
$$
the representation $\pi _s$ is associated to $l=\pm s^NZ^*$, and all
representations in $SI/Z(G)$ are obtained in this way. Theorem~\ref{MainThmUnif}
is thus proved completely. 
	\end{proof}

\section{Examples} \label{SectionExs}
This section provides three  examples that 
illustrate the main theorems. 

\begin{ex}
  Let $G$ be nilpotent with $\dim(G)\leq 6$. Then $G$ is graded with
  $\dim(Z(G) = 1 $ or $=2$. In the former case, the existence of an
  orthonormal basis in the orbit of a representation in $SI/Z(G)$
  follows from Theorems~\ref{MainThmQL} or~\ref{MainThmUnif}. In the
 latter case, the existence follows from Theorem~\ref{TVMainThm}: The conditions (i) --- (ii) are
satisfied because groups of low dimension have a relatively simple structure: Most
$SI/Z$-groups $G$ are a semidirect product of the form $G = \PID \rtimes \R^d $ of
an Abelian polarization $\PID$ and another Abelian subgroup. The
remaining ones have a normal Abelian and thus polarizing subgroup
$\PID$ and the quotient group $\PID \rquo G$ is also
Abelian.

This class of examples  was obtained in  H\"ofler's
thesis~\cite{Hoe14}. H\"ofler's work is based on the explicit formulas for
the irreducible representations of Nielsen~\cite{Nielsen} and could be done
by hand without substantial input from the theory of nilpotent
groups. Theorem~\ref{MainThmQL} or~\ref{MainThmUnif} offer a
systematic approach that carries over to a large class of nilpotent
groups. 

\end{ex}

	\begin{ex} [The Dynin-Folland Group $\HG{2}{1}$]
The first example of an $SI/Z$-group with a non-Abelian quotient group occurs in dimension
$7$, this is the so-called Dynin-Folland group $\HG{2}{1}$ as
in~\cite{DynFollGr}. As an illustration of  the abstract
construction of orthonormal bases we work out the details for this
example. 

The Dynin-Folland  group was introduced by Dynin~\cite{Dyn1} to study
pseudo-differential operators on the Heisenberg 
group and taken up again  by Folland~\cite{FollMeta} for his detailed study of so-called
meta-Heisenberg groups. Precisely, $\HG{2}{1}$ is a 
semidirect
product $\HG{2}{1} = \R ^4 \rtimes \mathbf{H}_1$ of $\R
^4$ and the $3$-dimensional Heisenberg group $\mathbf{H}_1$. Its $3$-step nilpotent Lie algebra $\HA{2}{1}$ is defined by a strong Malcev basis $\{ Z, Y_1, Y_2, Y_3, X_1, X_2,
X_3 \}$ with Lie brackets
	\begin{align}
	\left[\begin{array}{c|ccc|ccc}
		[\,.\,, \,.\,]&Y_1& Y_2 &Y_3 & X_1 & X_2 &X_3 \\ \hline
		Y_1 &&&&0&0&-Z \\
		Y_2 &&\textnormal{\Large{0}}&&0&-Z&0 \\
		Y_3 &&&&-Z&-\tfrac{1}{2} Y_1&\tfrac{1}{2} Y_2 \\ \hline
		X_1 &0&0&Z&0&0&0 \\
		X_2 &0&Z&\tfrac{1}{2} Y_1&0&0&-X_1 \\
		X_3 &Z&0&-\tfrac{1}{2} Y_2&0&X_1&0 \\
	\end{array}\right]. \label{LieBracketH21}
	\end{align}
The center of $\HA{2}{1}$ is $\Liez{\HA{2}{1}} = \R Z 
$ and is one-dimensional, the Heisenberg Lie algebra $\mathfrak{h}_1=
\{ X_1,X_2,X_3\} $ is a subalgebra of $\HA{2}{1}$ and the  complement
$\pid:= \Rspan{Z, Y_1, Y_2, Y_3}$ is Abelian and  an ideal of $\HA{2}{1}$. 

For $l = \l Z^* \in \Liez{g}^*$ with $\l \neq 0$,  the matrix representation of $B_l$ is given by
	\begin{align} \label{eq:c26}
	[B_l]
	=
	\left[\begin{array}{c|ccc|ccc}
		l([\,.\,, \,.\,])&Y_1& Y_2 &Y_3 & X_1 & X_2 &X_3 \\ \hline
		Y_1 &&&&0&0&-\l \\
		Y_2 &&\textnormal{\Large{0}}&&0&-\l&0 \\
		Y_3 &&&&-\l&0&0 \\ \hline
		X_1 &0&0&\l&&& \\
		X_2 &0&\l&0&&\textnormal{\Large{0}}& \\
		X_3 &\l&0&0&&& \\
	\end{array}\right].
	\end{align}
Hence, $B_l$ is non-degenerate with $\Pf(l) = \Abs{\l}^3$ and the
coadjoint orbit $\Orbit_l$ is the $6$-dimensional affine subspace $\l
Z^* + \Rspan{Y_1^*, \ldots, X_3^*}$. This implies that,  up to Plancherel measure zero, all coadjoint orbits are flat; these are precisely the orbits $\Orbit_{\l Z^*}, \l \in \R \setminus \{ 0 \}$.
Since $\pid$ is Abelian and of dimension $4$, it is a polarization for all $l \in \Liez{\HA{2}{1}}^*$; moreover, $\PID \rquo G \cong \mathbf{H}_1$. Consequently,
 all representations  $\pi_l \in SI/Z(\HG{2}{1})$ can be realized on
 $\L{2}{\mathbf{H}_1}$ in terms of the strong Malcev coordinates $\{
 Z, \ldots, X_3 \}$. 

Comparing \eqref{eq:c26} and \eqref{tildeB_l}, we set 
$\tilde{X}_3:= \l^{-1} X_1, \tilde{X}_2:= \l^{-1} X_2, \tilde{X}_1:=
\l^{-1} X_3$, then $\{\tilde X_3, \tilde X_2, \tilde X_1 \}$ is a basis that parametrizes the orbit $l\cdot
G$. Since we have only relabeled and rescaled the original basis, $\{Z, Y_1, Y_2, Y_3, \l^{-1} X_1, \l^{-1} X_2,
\l^{-1} X_3 \}$ is a strong Malcev basis and thus also a 
Ch-R-admissible basis  of $\Lie{g}$ subordinate to $l =
\l Z^*$.  The assumptions of  Theorem~\ref{TVMainThm} are satisfied and
the representation $\pi _l$ possesses an orthonormal basis associated
to the quasi-lattice $$\Gamma' :=
\pr \bigl( \exp(\R Y_1) \cdots \exp(\R \tilde{X}_1) \bigr) \subseteq
\HG{2}{1}/\exp(\R Z) \, . $$ 
If  $\l$ is rational, then the additional conditions (iii) - (v)  of
Theorem~\ref{TVMainThm} are satisfied, and there exists a uniform
subgroup $\Gamma $ of $\HG{2}{1}/\exp(\R Z)$, such that $\{ C \Abs{\l}^{-3/2} \pi _l
(\gamma ) 1_F \mid \gamma \in \Gamma \}$ is an orthonormal basis of
$\L{2}{\mathbf{H}_1}$.


In the following we perform the explicit
computations for the case of a uniform subgroup. We begin with the group law
 in strong Malcev coordinates. Let $g := \exp(z Z) \cdots
\exp(\tilde{x}_1 \tilde{X}_1)$ and $g' := \exp(z' Z) \cdots
\exp(\tilde{x}'_1 \tilde{X}_1)$ be two  
elements in $\HG{2}{1}$. A straight-forward computation based on the pairwise
commutation relations of their exponential factors yields 
	\begin{align}
		g g' &= \exp\Bigl( (z + z' + \frac{1}{\lambda} \sum_{j = 1}^3
                x_j y'_j -  \tfrac{1}{2\lambda ^2} \tilde{x}_1 \tilde{x}_2 y'_3) Z \Bigr)
                \exp\Bigl( (y_1 + y'_1 +  \tfrac{1}{2\lambda } \tilde{x}_2 y'_3)
                Y_1 \Bigr) \nn \\ 
		 &\hspace{15pt} \cdot \exp\Bigl( (y_2 + y'_2) Y_2 \Bigr) \exp\Bigl( (y_3 + y'_3) Y_3 \Bigr) \exp\Bigl( (\tilde{x}_3 + \tilde{x}'_3 + \l^{-1} \tilde{x}_1 \tilde{x}'_2) \tilde{X}_3 \Bigr) \exp\Bigl( (\tilde{x}_2 + \tilde{x}'_2) \tilde{X}_2 \Bigr) \nn \\
		  &\hspace{15pt} \cdot \exp\Bigl( (\tilde{x}_1 + \tilde{x}'_1) \tilde{X}_1 \Bigr). \nn
	\end{align}
The polynomial functions $P_Z, \ldots, P_{\tilde{X}_1}$ from Lemma~\ref{LemPropStrongMC} are thus given by
\begin{equation}
\left\{\begin{array}{rcl}
	P_Z(z, \ldots, \tilde{x}'_1) &=& z + z' + \frac{1}{\l} \sum_{j
          = 1}^3  \tilde{x}_j y'_j - \frac{1}{2 \l^2} \, \tilde{x}_1 \tilde{x}_2 y'_3, \\
	P_{Y_1}(z, \ldots, \tilde{x}'_1) &=& y_1 + y'_1 + \frac{1}{2 \l} \, \tilde{x}_2 y'_3,\\
	P_{Y_2}(z, \ldots, \tilde{x}'_1) &=& y_2 + y'_2 , \\
	P_{Y_3}(z, \ldots, \tilde{x}'_1) &=& y_3 + y'_3 , \\
	P_{\tilde{X}_3}(z, \ldots, \tilde{x}'_1) &=& \tilde{x}_3 +
        \tilde{x}'_3 + \frac{1}{\l} \, \tilde{x}_1 \tilde{x}'_2 , \\ 
	P_{\tilde{X}_2}(z, \ldots, \tilde{x}'_1) &=& \tilde{x}_2 + \tilde{x}'_2, \\
	P_{\tilde{X}_1}(z, \ldots, \tilde{x}'_1) &=& \tilde{x}_1 +
        \tilde{x}'_1 . \\
\end{array}\right. \nn
\end{equation} 
 The coefficients of these polynomials are  $1, \tfrac{1}{\l}, \tfrac{1}{2 \l}, \tfrac{1}{2\l^2}$. For rational
$\l = p/q \in \mathbb{Q}$ with $p,q$ relatively prime integers,
these are the numbers $1, q/p, q/2p$ and $q^2/2p^2$. According to the
construction of the uniform subgroup in Proposition~\ref{Subgroup}, we must choose $K \in \N $ such that all
denominators divide $K$; in this case $K=2p^2$ suffices. We then
employ the change of basis from $\{ Z, \ldots, \tilde{X}_1 \}$ to $\{
K^{-16} Z, K^{-8} Y_1, K^{-4} Y_2, K^{-2} Y_3, K^2 \tilde{X}_3, K^4
\tilde{X}_2, K^8 \tilde{X}_1 \}$. In these coordinates the  group multiplication is described
by the polynomials $P'_Z, \ldots, P'_{\tilde{X}_1}$ given by 
\begin{equation}
\left\{\begin{array}{rcl}
	P'_Z(z, \ldots, \tilde{x}'_1) &=& z + z' + \frac{K^{16}}{\l}
        \sum_{j = 1}^3  \tilde{x}_j y'_j - \frac{K^{26}}{2 \l^2} \, \tilde{x}_1 \tilde{x}_2 y'_3, \\
	P'_{Y_1}(z, \ldots, \tilde{x}'_1) &=& y_1 + y'_1 + \frac{K^{10}}{2 \l} \, \tilde{x}_2 y'_3,\\
	P'_{Y_2}(z, \ldots, \tilde{x}'_1) &=& y_2 + y'_2 ,\\
	P'_{Y_3}(z, \ldots, \tilde{x}'_1) &=& y_3 + y'_3 ,  \\
	P'_{\tilde{X}_3}(z, \ldots, \tilde{x}'_1) &=& \tilde{x}_3 +
        \tilde{x}'_3 + \frac{K^{10}}{ \l} \tilde{x}_1 \tilde{x}'_2 , \\ 
	P'_{\tilde{X}_2}(z, \ldots, \tilde{x}'_1) &=& \tilde{x}_2 + \tilde{x}'_2, \\
	P'_{\tilde{X}_1}(z, \ldots, \tilde{x}'_1) &=& \tilde{x}_1 + \tilde{x}'_1. \\
\end{array}\right. \nn
\end{equation} 
Moreover,
	\begin{align*}
		\Gamma_G :=& \bigl\{ \exp(\zeta K^{-16} Z) \exp(\vartheta_1 K^{-8} Y_1) \exp(\vartheta_2 K^{-4} Y_2) \exp(\vartheta_3 K^{-2} Y_3) \\
		&\hspace{3pt} \exp(\eta_3 K^2 \tilde{X}_3) \exp(\eta_2
                K^4 \tilde{X}_2) \exp(\eta_1 K^8 \tilde{X}_1) \mid
               \zeta,  \vartheta_1, \ldots, \eta_3 \in \Z  \bigr\}
	\end{align*}
forms a uniform subgroup of $\HG{2}{1}$ with fundamental domain
	\begin{align*}
		\Sigma_G :=& \bigl\{ \exp(\tilde{x}_1 K^8
                \tilde{X}_1) 
   \exp(\tilde{x}_2 K^4 \tilde{X}_2)  \exp(\tilde{x}_3 K^2
   \tilde{X}_3) \\ 
&  \exp(y_3 K^{-2} Y_3)  \exp(y_2 K^{-4}Y_2) \exp(y_1 K^{-8} Y_1) \exp(z K^{-16} Z) \mid z,  y_1, \ldots, \tilde{x}_3 \in [0, 1)
                \bigr\}, 
	\end{align*}
so we pick $\Gamma := \pr(\Gamma_G)$ with fundamental domain $\Sigma
:= \pr(\Sigma_G)\subseteq \HG{2}{1}/\R Z $ to construct an  orthonormal basis.

We now realize $\pi$ in terms of the new Malcev coordinates. Since
$\mathbf{H}_1$ is a subgroup of $\HG{2}{1}$, the natural splitting $g
= p(g) s(g) = mh$ is satisfied by
 $$m = \exp(z K^{-16} Z) \exp(y_1 K^{-8} Y_1)
\exp(y_2 K^{-4} Y_2) \exp(y_3 K^{-2} Y_3)$$
 and $$h = \exp(\tilde{x}_3
K^2 \tilde{X}_3) \exp(\tilde{x}_2 K^4 \tilde{X}_2) \exp(\tilde{x}_1
K^8 \tilde{X}_1)\, .
$$ 
Following~\eqref{ab}, for a given  $h' := \exp(t_3 K^2
\tilde{X}_3) \exp(t_2 K^4 \tilde{X}_2) \exp(t_1 K^8 \tilde{X}_1)$ we
compute  
	\begin{align*}
		h' g &= h' m h'^{-1} h' h \\
		&= \exp\Bigl( \bigl( z + \frac{K^{16}}{\l} \sum_{j =
                  1}^3 t_j y_j - \frac{K^{26}}{2 \l^2} \, t_1 t_2 y_3
                \bigr) K^{-16} Z \Bigr) \exp\Bigl( \bigl( t_1 +
                \frac{K^{10}}{2 \l} \, t_2 y_3 \bigr) K^{-8}Y_1 \Bigr)
                \\ 
		&\hspace{15pt} \cdot \exp\Bigl( y_2 K^{-4} Y_2 \Bigr) \exp\Bigl( y_3 K^{-2} Y_3 \Bigr) \\
		&\hspace{15pt} \cdot \exp\Bigl( \bigl( t_3 + \tilde{x}_3 + \frac{K^{10}}{ \l} t_1 \tilde{x}_2 \bigr) K^2 \tilde{X}_3 \Bigr) \exp\Bigl( (t_2 + \tilde{x}_2) K^4 \tilde{X}_2 \Bigr) \exp\Bigl( (t_1 + \tilde{x}_1) K^2\tilde{X}_1 \Bigr) \\
		 &= m'' \, h'h \, .
	\end{align*}
Since $\bracket{\lambda Z^*}{\log \big( h' m {h'}\inv \big)} =
\bracket{\lambda \mathrm{Ad}^*({h'}\inv) Z^*}{\log (m)}$, the
coefficients of the  $y_j$'s in the $Z$-component yield the 
polynomial functions describing the 
orbit $l \cdot G$ (cf.~Theorem~\ref{AdaptedChevalleyRosenlicht}) 
\begin{equation}
\left\{\begin{array}{rclcl}
		Q_3(t_3, t_2, t_1) &=& t_3 + \tilde{Q}_3(t_2, t_1) &=& t_3 - \frac{K^{10}}{2 \l^2} \, t_1 t_2, \\
		Q_2(t_2, t_1) &=& t_2 + \tilde{Q}_2(t_1) &=& t_2, \\
		Q_1(t_1) &=& t_1. &&
\end{array}\right. \nn
\end{equation}

We denote the coordinates of $m'' := p(h' g)$ by $z'', y''_1, y''_2,
y''_3$. Since $\mathbf{H}_1$ is a subgroup, $q(h' h) = h' h$ and $p(h'
h) = e$, and the commutativity of $\pid$ implies 
	\begin{align*}
		\log(m'') = z'' Z * y''_1 Y_1 * y''_2 Y_2 * y''_3 Y_3 = z'' Z + y''_1 Y_1 + y''_2 Y_2 + y''_3 Y_3.
	\end{align*}
Hence, the action of  $\pi_l$ on a function $f\in L^2(\mathbf{H}_1)$  is explicitly given by
	\begin{align}
		\big( \pi_l(mh) f \big)(h') &= e^{2 \pi i
                  \bracket{\l Z^*}{\log ( h' m {h'} \inv )}}
                \hspace{2pt} e^{2 \pi i \bracket{\l Z^*}{\log ( p(h'
                    h) )}} \hspace{2pt} f(h'h) \notag \\ 
		&= e^{2 \pi i (\l K^{-16} z + \sum_{j=1}^3 t_j y_j -
                  \frac{K^{10}}{2 \l } \, t_1 t_2
                  y_3)} \label{DynFollRep1}  \\
		&\hspace{10pt} \times f\biggl( \exp \Bigl( \bigl( t_3
                + \tilde{x}_3 + \frac{K^{10}}{ \l} t_1 \tilde{x}_2
                \bigr) K^2 \tilde{X}_3 \Bigr) \exp \Bigl( (t_2 +
                \tilde{x}_2) \tilde{X}_2 \Bigr) \exp \Bigl( (t_1 +
                \tilde{x}_1) \tilde{X}_1 \Bigr) \biggr). \notag
	\end{align}


In this set of Malcev coordinates, the orthonormal basis 
$ \bigl\{ K^{-7} \hspace{2pt} \Abs{\l}^{-3/2} \hspace{2pt} \pi(\gamma) 1_\PFD \mid \in \Gamma \bigr\}$
 of $\L{2}{\mathbf{H}_1}$ becomes the explicit orthonormal basis
	\begin{align} \label{eq:c34}
 		\Bigl \{ e^{2 \pi i (\sum_{j=1}^3 t_j \vartheta_j
                  -                 \frac{K^{10}}{2 \l } \, t_1 t_2 \vartheta_3)}
                \hspace{2pt} 1_{[0,1)^3} \bigl( t_3 + \eta_3 +
                \frac{K^{10}}{ \l} t_1 \eta_2 , t_2 + \eta_2, t_1 +
                \eta_1 \bigr) \mid \vartheta_1, \ldots, \eta_3 \in \Z
                \Bigr \}  
	\end{align}
 of $\L{2}{\R^3} \cong L^2(\mathbf{H}_1)$.
Of course, once such an explicit formula is derived, it is not difficult
to check directly that \eqref{eq:c34} is an orthonormal basis.

 We  discuss briefly  another construction of an  orthonormal basis of
 $\L{2}{\R^3}$. This construction  is  based on a gradation of $
 \HA{2}{1}$ and Theorem~\ref{MainThmQL}. For contrast we  treat the quasi-lattice case. 
 We observe that $\HA{2}{1}$ admits the
 gradation 
$\HA{2}{1} = \ggl _3 \oplus \ggl _2 \oplus \ggl _1$
with 
        \begin{align*}
        	\ggl _3 := \R Z, \hspace{10pt} \ggl _2 := \Rspan{Y_1,
                  Y_2, X_1}, \hspace{10pt} \ggl _1 := \Rspan{Y_3, X_2, X_3}.
        \end{align*}

Setting  $X'_1 := Y_3$ and $ Y'_3 := X_1$, the table of the Lie
algebra is 
	\begin{align}
	\left[\begin{array}{c|ccc|ccc}
		[\,.\,, \,.\,]&Y_1& Y_2 &Y'_3 & X'_1 & X_2 &X_3 \\ \hline
		Y_1 &&&&0&0&-Z \\
		Y_2 &&\textnormal{\Large{0}}&&0&-Z&0 \\
		Y'_3 &&&&Z&0&0 \\ \hline
		X'_1 &0&0&-Z&0&-\tfrac{1}{2} Y_1&\tfrac{1}{2} Y_2 \\
		X_2 &0&Z&0&\tfrac{1}{2} Y_1&0&-Y'_3 \\
		X_3 &Z&0&0&-\tfrac{1}{2} Y_2&Y'_3&0 \\
	\end{array}\right]. \label{ChVLieBracketH21}
	\end{align}
We notice that $\pid := \Rspan{Z, Y_1, Y_2, Y'_3}= \ggl _3 \oplus \ggl
_2$ is Abelian and that $\ggl / \pid $ is Abelian  and thus $M \rquo G \cong
(\R ^3, +)$, but the complement $\qa = \Rspan{X'_1, X_2, X_3}$ is not
a subalgebra. However, for $l = \l Z^*$, $\l \neq 0$, the polynomial functions $Q_j$ which parametrize the orbit $l \cdot G$ are much simpler,
namely of order $1$, as we deduce from the lower left block in
\eqref{ChVLieBracketH21}. Precisely, we choose $\tilde{X}'_3 := - \l^{-1}
X'_1, \tilde{X}_2 := \l^{-1} X_2, \tilde{X}_1 := \l^{-1} X_3$ as a
basis for the parametrization.  
Theorem~\ref{MainThmQL}  applies and asserts the existence of  an
orthonormal basis of $\L{2}{\PID \rquo \HG{2}{1}}$  arising from the
quasi-lattice 
	\begin{align*}
		\Gamma'_{\HG{2}{1}} := \exp(\Z Z) \exp(\Z Y_1) \exp(\Z Y_2) \exp(\Z Y'_3) \exp(\Z \tilde{X'}_3) \exp(\Z \tilde{X}_2) \exp(\Z \tilde{X}_1) \subseteq \HG{2}{1}.
	\end{align*}

To  be more specific, let $\pi _l'$ be the unitary representation
corresponding to $l= \lambda Z^*$, but now induced from the
polarization  $\pid = \ggl _3 \oplus \ggl _2$. By the general theory
$\pi _l'$ is equivalent to $\pi _l$ in~\eqref{DynFollRep1}.
 This  polarization leads to a simpler formula for the
 representation. 
Indeed, the first exponent of 
	\begin{align*}
		\Bigl( \pi'_l(mh) f \Bigr) \bigl(q(h') \bigr) &= e^{2 \pi i \bracket{\l Z^*}{\log \bigl( h' m (h')\inv \bigr)}} \hspace{2pt} e^{2 \pi i \bracket{\l Z^*}{\log \bigl( p(h' h) \bigr)}} \hspace{2pt} f(h'h)
	\end{align*}
is a sum of first-order polynomial functions, and  the translation
$f(h' h)$ is Abelian, because of $[\qa, \qa] \subseteq \Rspan{Y_1, Y_2,
  Y_3}$; the $Z$-component of the non-vanishing
factor $$p(h' h) = \exp\bigl( (\tfrac{t_1 t_2 \tilde{x}_3}{2 \l^2} +
\tfrac{t_1 \tilde{x}_2 (t_3 + \tilde{x}_3)}{\l^2}) Z \bigr) \exp
\bigl(-\tfrac{t_2 \tilde{x}_3}{2 \l^2} Y_1 \bigr) \bigl(\tfrac{t_1
  \tilde{x}_3}{2 \l^2} Y_2 \bigr) \exp\bigl( \tfrac{t_1
  \tilde{x}_2}{\l^2} Y'_3 \bigr) \in \PID$$ enters the second
exponent. All in all we have
	\begin{align*}
		\Bigl( \pi'_l(mh) f \Bigr) \bigl(q(h') \bigr) &= e^{2
                  \pi i (\l z + \sum_{j=1}^3 t_j y_j)} \hspace{2pt}
                e^{2 \pi i \bigl( \tfrac{t_1 t_2 \tilde{x}_3}{2 \l} +
                  \tfrac{t_1 \tilde{x}_2 (t_3 + \tilde{x}_3)}{\l}
                  \bigr)} \\ 
		& \hspace{15pt} f\Bigl( \exp\bigl( (t_3 + \tilde{x}_3) \tilde{X}_3 \bigr) \exp\bigl( (t_2 + \tilde{x}_2) \tilde{X}_2 \bigr) \exp\bigl( (t_1 + \tilde{x}_1) \tilde{X}_1 \bigr) \Bigr).
	\end{align*}

We therefore obtain the orthonormal basis $\{  \Abs{\l}^{-3/2} \pi
_l'(\vartheta \eta \inv ) 1_\PFD \mid \vartheta \eta \in \Gamma \}$ or 
	\begin{align}
		\Bigl\{ e^{2     \pi i \sum_{j=1}^3 t_j \vartheta_j} \hspace{2pt}
                e^{2 \pi i \bigl( \tfrac{t_1 t_2 \eta_3}{2 \l} +
                  \tfrac{t_1 \eta_2 (t_3 + \eta _3)}{\l}
                  \bigr)}  1_{[0,1)}(t_1 + \eta_1, t_2 + \eta_2, t_3
  + \eta_3)   \mid \vartheta _j,  \eta _j \in \Z  \Bigr\} \label{DFonb2}
	\end{align} 
 of $\L{2}{\R^3}$. 
This construction resembles the Gabor orthonormal basis for $L^2(\R
^3)$ from Section~2.1. Again, once \eqref{DFonb2} is written down, the
orthonormality and the completeness are easy to check. The point here
is that the appropriate choice of coordinates and polarization are
guided by a theory and are rather difficult to guess by hand. 
\end{ex}

	\begin{rem}
This example can be generalized to an infinite family of graded $3$-step nilpotent
groups $\HG{2}{j}$ of dimension $4j+3$. Technically, meta-Heisenberg groups $G$ such as $\HG{2}{j}$ arise from the Lie
algebra $\Lie{g}$ generated by a basis of
left-invariant vector fields of a $2$-step nilpotent group $\bar{G}$ and multiplication by first-order
polynomials on $\bar{\Lie{g}} \cong \R^{d}$ in the variables determined by
the basis. In the case $G = \HG{2}{j}$ we have $\bar{\Lie{g}} = \mathbf{H}_j \cong \R^{2j +1} =: \R^d$. See~\cite{DynFollGr} for a detailed analysis of the
Dynin-Folland groups.
	\end{rem}

	\begin{ex}
The following $7$-dimensional Lie algebra is a slight variation of the
Lie algebra in \cite[Rem.~3.1.6]{FiRuz}. It is an example for which
only the quasi-lattice version of the technical
Theorem~\ref{TVMainThm} yields an orthonormal basis
of $\L{2}{\R^3}$, but not its special cases for graded groups. 

Let $\{X_1, \ldots, X_7\}$ be the Euclidean standard basis for $\Lie{g} := \R^7$ endowed with the Lie bracket relations
	\begin{align*}
		[X_1, X_j] = X_{j+1} \hspace{10pt} \mbox{ for } j = 2, \ldots, 6, \hspace{10pt} [X_2, X_3] = \sqrt{2} \hspace{2pt} X_6, \\
		[X_2, X_4] = [X_5, X_2] = [X_3, X_4] = X_7.
	\end{align*}
As in \cite[Rem.~3.1.6]{FiRuz} 
these Lie brackets define a $7$-dimensional Lie
algebra with $1$-dimensional center 
that cannot be equipped with any gradation.
The bracket $[X_2,X_3]= \sqrt{2} X_6$ excludes the existence of a
rational structure. 
 Consequently  we cannot apply Theorems~\ref{MainThmQL} or
\ref{MainThmUnif}. 

Nevertheless, we can apply  Theorem~\ref{TVMainThm}. Indeed, the basis $\{ X_7, \ldots, X_1 \}$ is a strong Malcev basis of $\Lie{g}$ and $\pid := \Rspan{X_7, X_6, X_5, X_4}$ is a polarizing ideal for every $l := \l X_7^*$, $\l \in \R \setminus \{ 0 \}$. Since
	\begin{align*}
	[B_l]
	=
	\left[\begin{array}{c|ccc|ccc}
		l([\,.\,, \,.\,])&X_6& X_5 &X_4&X_3& X_2& X_1 \\ \hline
		X_6&&&&0&0&-\l \\
		X_5&&\text{\Large{0}}&&0&\l&0 \\
		X_4&&&&-\l&-\l&0 \\ \hline
		X_3&0&0&\l&&& \\
		X_2&0&-\l&\l&&\text{\Large{0}}& \\
		X_1&\l&0&0&&& \\
	\end{array}\right],
	\end{align*}
the basis $\{ Z, Y_1, \ldots, \tilde{X}_1 \} := \{ X_7, X_6, X_5, X_4, \l \inv X_3, -\l \inv X_2, \l \inv X_1 \}$ is Ch-R-admissible subordinate to $l$. Thus, conditions (i)---(ii) of Theorem~\ref{TVMainThm} are satisfied and we obtain an orthonormal basis for $\L{2}{\PID \rquo G} \cong \L{2}{\R^3}$ derived from a quasi-lattice.


	\end{ex}

\end{document}